\documentclass[reqno,11pt]{amsart}

\usepackage{amssymb,graphics,graphicx,wrapfig}

\def\eps{\varepsilon}

\def\be{\begin{equation}}
\def\ee{\end{equation}}
\def\ba{\begin{align}}
\def\bm{\begin{multline}}
\def\bfig{\begin{figure}[htb]}
\def\efig{\end{figure}}

\numberwithin{equation}{section}
\newtheorem{theorem}{Theorem}[section]
\newtheorem{proposition}[theorem]{Proposition}
\newtheorem{lemma}[theorem]{Lemma}

\DeclareMathSymbol{\leqslant}{\mathalpha}{AMSa}{"36}
\DeclareMathSymbol{\geqslant}{\mathalpha}{AMSa}{"3E}
\DeclareMathSymbol{\doteqdot}{\mathalpha}{AMSa}{"2B}
\DeclareMathSymbol{\circlearrowright}{\mathalpha}{AMSa}{"08}
\DeclareMathSymbol{\subsetneq}{\mathalpha}{AMSb}{"28}
\DeclareMathSymbol{\supsetneq}{\mathalpha}{AMSb}{"29}
\renewcommand{\leq}{\;\leqslant\;}
\renewcommand{\geq}{\;\geqslant\;}

\newcommand{\dd}{{\rm d}}
\newcommand{\e}[1]{\,{\rm e}^{#1}\,}

\newcommand{\upchi}{\raise 2pt \hbox{$\chi$}}

\def\writefig#1 #2 #3 {\rlap{\kern #1 truecm \raise #2 truecm
\hbox{#3}}}

\newcommand{\caS}{{\mathcal S}}

\newcommand{\bbE}{{\mathbb E}}

\newcommand{\bbN}{{\mathbb N}}

\newcommand{\bbP}{{\mathbb P}}

\newcommand{\bbR}{{\mathbb R}}

\newcommand{\bbZ}{{\mathbb Z}}

\newcommand{\bsv}{{\boldsymbol v}}

\newcommand{\bsx}{{\boldsymbol x}}

\begin{document}


\title{Random permutations of a regular lattice}

\author{Volker Betz}
\address{Volker Betz \hfill\newline
\indent Department of Mathematics, \hfill\newline 
\indent University of Warwick \hfill\newline
\indent and \hfill\newline
\indent FB Mathematik, TU Darmstadt \hfill\newline
{\small\rm\indent http://www.mathematik.tu-darmstadt.de/$\sim$betz/} 
}
\email{betz@mathematik.tu-darmstadt.de}

\maketitle

\begin{quote}
{\small
{\bf Abstract.}
Spatial random permutations were originally studied due to their connections to 
Bose-Einstein condensation, but they possess many interesting properties of their own. 
For random permutations of a regular lattice with periodic boundary conditions, 
we prove existence of the infinite volume limit under fairly weak assumptions. 
When the dimension of the lattice is two, we give numerical evidence of a 
Kosterlitz-Thouless transition, and of long cycles having an almost sure fractal dimension 
in the scaling limit. Finally we comment on possible connections 
to Schramm-L\"owner curves. 
}  

\vspace{1mm}
\noindent
{\footnotesize {\it Keywords:} Random permutations, infinite volume limit, Kosterlitz-Thouless transition,
fractal dimension, Schramm-L\"owner evolution}

\vspace{1mm}
\noindent
{\footnotesize {\it 2000 Math.\ Subj.\ Class.:} 28A80, 60K35, 60D05, 82B26, 82B80}
\end{quote}

\vspace{3mm}
\centerline{ \em Dedicated to Herbert Spohn on the occasion of his 65th birthday} 
\vspace{3mm}


\section{Introduction}

Spatial random permutations (SRP) are a simple but intriguing probabilistic
model with a distinct statistical mechanics flavor. They are easy to 
describe and to visualize. Pick a set $X_N$ of $N$ points in $\bbR^d$, 
which in the present paper will always be a subset of a regular lattice. 
A SRP measure is a probability measure 
on the set of bijective maps (permutations) $X_N \to X_N$ that respects the 
location of the points $x \in X_N$. The simplest version, and the one we will
exclusively treat here, is the independent jump penalization. We fix $\alpha > 0$ and a
jump energy function $\xi : \bbR^d \to \bbR$. To fix the ideas, we recommend to always think of the 
physically most relevant case $\xi(x) = |x|^2$. The vector $\pi(x)-x$ will be called the 
{\em jump} originating in $x$ of the permutation $\pi$. We assign a Boltzmann type weight 
$\e{-\alpha \xi(\pi(x)-x)}$ to each jump. 
The weight of $\pi$ is the product of all its jump weights; normalizing the weights so that their sum 
equals one yields the probability measure.  
Precise formulae and more discussion
will be given in Section \ref{basics}. A convenient visualization of a typical SRP 
can be found in Figure \ref{fig1}. 

We will always assume that $\xi$ is bounded below and that $\xi(x) \to + \infty$ 
as $|x| \to \infty$. Let us for now that assume in addition that $\xi$ is minimal 
at $x=0$. Then, jumps of typical SRP are short, and the  
parameter $\alpha$ determines the balance between the energy 
$\sum_{x \in X} \xi(\pi(x)-x)$ which is minimized for the identity permutation, 
and the entropy which grows with the number of jump targets available to 
each $x$. We are mainly interested in properties of the infinite volume limit $N \to \infty$, 
which is a measure on permutations of an infinite regular lattice $X$. 
A first question is whether such an infinite 
volume limit exists. For a regular lattice with periodic boundary conditions, we give an 
affirmative answer 
under mild conditions on $\xi$  in Theorem \ref{existence}. 

\bfig
\begin{center}
\includegraphics[width= 0.5\textwidth]{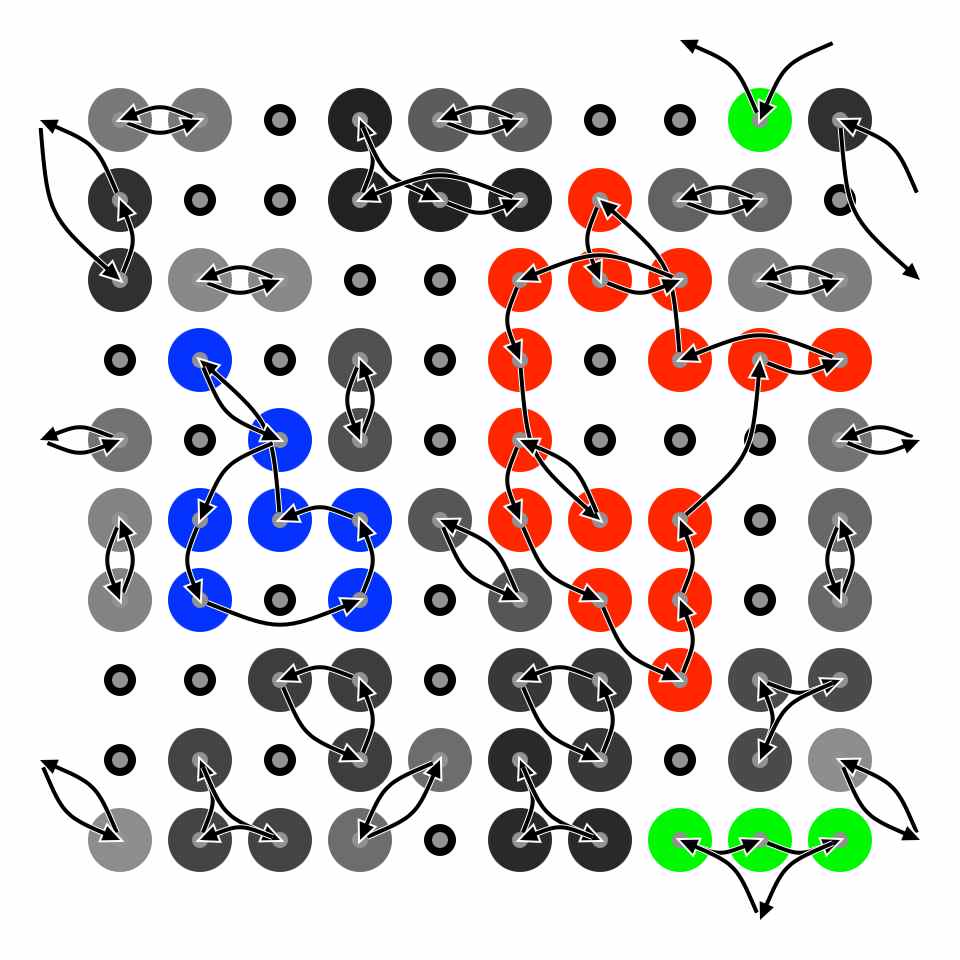} 
\end{center}
\caption{Image of a spatial random permutation on a 
square grid of $100$ points with periodic 
boundary conditions. An arrow is drawn from each point $x \in X$ to its image 
$\pi(x)$. Points in the same cycle are of the same color. 
Fixed points of $\pi$ have no arrow. 
}
\label{fig1}
\efig

The most interesting properties of spatial random permutations are properties of their cycles. 
For $x \in X$, write $C_x(\pi) = \{ \pi^j(x): j \in \bbZ \}$ for the (finite or infinite) 
cycle that contains $x$.
What is the expected size of $C_x(\pi)$?  
When $\alpha$ is very large, the overwhelming majority 
of points will be fixed points of a typical SRP $\pi$, 
and the cycles will be very short. 
When $\alpha$ is small, long jumps are 
only weakly penalized, and we can expect cycles to be long or even infinite. 
This raises the question whether there is a phase transition, i.e.\ a 
value $\alpha_c>0$ such that for a given point $x \in X$, 
$\bbP(|C_x(\pi)| < \infty) > 0$ if and only if $\alpha < \alpha_c$. 
For lattice permutations, there is no rigorous answer, 
but \cite{GRU} contains detailed numerical investigations. There it is found that for $\xi(x) = |x|^2$, 
a phase transition exists if and only if the dimension of the lattice is strictly greater than two.

Absence of a phase transition in two dimensions is known in systems with a 
continuous symmetry, such as the XY-model, where by 
the Mermin-Wagner theorem the continuous symmetry cannot be broken \cite{MW66}. 
Instead, a Kosterlitz-Thouless phase transition occurs, 
where the decay of correlations changes from exponential to algebraic. 
There is no obvious continuous symmetry present in spatial random permutations.
Nevertheless, in Section \ref{section KT} we give numerical 
evidence of a Kosterlitz-Thouless phase transition for SRP of a planar lattice.
We also discuss, on a heuristic level, similarities and differences between SRP
and the plane rotor model. 
In the physics literature, a connection is made between the XY-model and two-dimensional 
superfluids \cite{NK77}. By the connection of SRP to bosonic many-particle systems (see below),
a link of planar SRP and the XY model is somewhat plausible. The precise nature of this link is unclear. 

A natural step beyond studying the length distribution of SRP cycles is to investigate their geometry
in the scaling limit. In three dimensions, numerical studies \cite{GLU12} show that the points belonging
to long cycles are uniformly distributed in space, and the same is expected to hold for higher 
dimensions. Thus only the case of two dimensions can give rise to a nontrivial geometry. On the other
hand, by our results of Section \ref{section KT}, for any given point in space,
the cycle containing that point will be finite with probability one after taking the infinite
volume limit, and thus shrink to a point in the subsequent scaling limit. We therefore have to 
condition on the event that the origin is in a cycle whose length diverges in the infinite 
volume limit. Numerically, this can be approximately achieved by either studying the longest 
cycle in a large finite volume, or by forcing a cycle through the system by the 
boundary conditions. The latter is much more practical for large $\alpha$. 
We find strong numerical evidence that in the scaling limit, 
for $\alpha$ smaller than the critical value 
$\alpha_c$ at which the Kosterlitz-Thouless transition happens, 
long cycles have an almost sure fractal dimension, 
the latter depending linearly on $\alpha$. 
For $\alpha > \alpha_c$, our numerical procedures still suggest an 
almost sure fractal dimension, this 
time with a power law dependence on $\alpha$. However, due to strong metastability
effects and since
we are conditioning on very rare events, 
our numerics become rather unreliable there.
At the end of Section 6, we include a short discussion about
possible interpretations of our results for large $\alpha$. 

In the final section of this paper, we speculate about 
the question whether the scaling 
limit of long SRP cycles in two dimensions has any connection to the Schramm-L\"owner curves
\cite{Law05}. We give two arguments in favor of this connection and one possible objection. 
Ultimately, the question remains unanswered and will require further investigation. 

The present paper treats spatial random permutations 
from the point of view of classical statistical mechanics. 
An important connection that we do not treat here is to the mathematics of the free and the interacting 
Bose gas. Indeed, this is the context which motivated the model originally. 
In \cite{Fey}, Feynman argued that in a model somewhat  
related to ours, the occurrence of infinite cycles signifies  
the onset of Bose-Einstein condensation in the interacting Bose gas. No significant mathematical 
progress has been made on these ideas for interacting Bosons, 
but for the free Bose gas, S\"ut\H o \cite{Suto1,Suto2} put Feynmans ideas onto firm mathematical ground. 
In \cite{BU1}, S\"ut\H o's ideas were generalized and put into a probabilistic context. 
This leads to the annealed model of SRP, where the points $x$ are not confined on a lattice, and 
an averaging over all possible positions of the points $x$ in the finite volume takes place when 
computing the weight of a given permutation. In that model, and some generalizations of it \cite{BU2, BU4}, 
a phase transition of the type observed in \cite{GRU} is proved under suitable conditions.
For most reasonable choices of $\xi$, these conditions imply that the space dimension  
is three or greater. 

It would be very desirable to establish this kind of result also for lattice SRP, 
but the methods used  in \cite{BU2, BU4} 
make critical use of the averaging over the points $x$ and 
cannot be adapted to the lattice situation. Indeed, apart from the results 
presented in Section \ref{ivl} below, all that has been proved about lattice
SRP can be found in \cite{BU1} and \cite{BiRi_unplublished}. The relevant sections
of the former reference establish a criterion for the existence of the 
infinite volume limit, which we will use below, and prove the absence of infinite 
cycles for sufficiently large values of $\alpha$. 
The latter work gives a rather complete description of the
set of all infinite volume SRP measures, but only in the case of the one-dimensional,
not necessarily regular, lattice. Nothing at all is rigorously known 
about cycle properties or phase transitions for 
higher dimensional SRP models.
It is my hope that some of the numerical results presented below will 
convince the reader that it is worthwhile putting some energy into changing 
this situation. 

\medskip
\noindent
{\bf Acknowledgments.}
I wish to thank Daniel Ueltschi for introducing me to the topic of SRP, 
for many good discussions, and for useful comments on the present paper; 
Thomas Richthammer for letting me have the manuscript
\cite{BiRi_unplublished} prior to its publication; 
and Alan Hammond for useful discussions on SLE. 
Finally I wish to thank one of the anonymous referees of this paper for his/her 
detailed and insightful comments and suggestions.

\section{Spatial random permutations on a periodic regular lattice} \label{basics}

Here we give the precise definitions of our model. Later we will focus on 
two-dimensional SRP, but in the present and the next section the space dimension is arbitrary. 

Let $X \subset \bbR^d$ be a regular lattice, i.e.\ pick a basis $v_i$, $1 \leq i \leq d$ of $\bbR^d$ and 
let $X = \{ \sum_{i=1}^d z_i v_i: z_i \in \bbZ \}$. First we consider finite subsets of $X$. Usually 
(e.g.\ in \cite{GRU, GLU12}), this is just done by intersecting $X$ with a finite 
cube $\Lambda$. In our case, this approach works for a cubic lattice, but we need periodic boundary 
conditions and thus use a more specific choice of finite volume approximation 
for general lattices. 

For $L \in \bbN$, $X_L$ denotes the subset of $X$ with side length 
$L$ and periodic boundary conditions, centered around $0$. 
Explicitly, $X_L = \{ \sum_{i=1}^d z_i v_i: z_i \in \bbZ_L \}$, where 
$\bbZ_L = \bbZ \cap [-L/2, L/2)$. For $z,w \in \bbZ_L$, the periodic difference 
\[
(z-w)_{\rm per} = \begin{cases} 
z-w & \text{if } z-w \in \bbZ_L \\
z-w+L & \text{if } z-w \in [-L,-L/2), \\
z-w-L & \text{if } z-w \in [L/2,L).
\end{cases}
\]
is again an element of $\bbZ_L$, and thus for $x = \sum z_i v_i$ and $y = \sum w_i v_i$, 
\[
(x-y)_{\rm per} = \sum_{i=1}^d (z_i-w_i)_{\rm per} v_i
\]
is an element of $X_L$. We will always use the periodic difference on $X_L$, and will simply write $x-y$ 
instead of $(x-y)_{\rm per}$. We will also consistently write $N = N(L) = |X_L|$ to denote the number of 
points in $X_L$. 

Let $\caS(X_L)$ denote the set of permutations on $X_L$, and $\caS(X)$ the set of permutations on $X$. A 
{\em jump energy} is a function $\xi: X \to \bbR$ which is bounded from below. 
The measure of spatial random permutations on the torus $X_L$
with jump energy $\xi$ and parameter $\alpha$
is the probability measure $\bbP_N$ on $\caS(X_L)$ such that the probability of a permutation 
$\pi: X_L \to X_L$ is given by a Boltzmann weight with total energy 
\be \label{energy}
H_N(\pi) = \sum_{x \in X_L} \xi (\pi(x) - x).
\ee
Explicitly, we have  
\be \label{measure}
\bbP_N (\pi) = \frac{1}{Z_N} \exp \big( - \alpha H_N(\pi) \big),
\ee
where $Z_N$ is the normalization. 
We call the vector $\pi(x)-x$ the {\em jump} starting from $x$, and identify 
$\bbP_N$ with the measure on 
$\caS(X)$ where jumps starting from outside of $X_L$ have length zero. 

Usually $\xi$ grows at infinity, so that it is energetically favorable when a permutation $\pi$ has 
few large jumps. While in the later sections we will almost exclusively 
deal with the cases $X = \bbZ^2$ and $\xi(x) = |x|^2$, here and in Section \ref{ivl} 
we allow $\xi$ to be more general and in particular non-symmetric and non-convex. 

Equation \eqref{measure} is  reminiscent of a finite volume Gibbs measure. In that context, 
it would be natural to call $\alpha$ the inverse temperature and denote it by $\beta$ as usual. 
The problem with
this is that in its annealed version, the model of spatial random permutations with 
$\xi(x) = |x|^2$ is closely connected with 
Bose-Einstein condensation, see \cite{BU1}.  
But then, $\beta$ would be proportional to the {\em temperature} of the Bose gas.  
The reason is that the integral 
kernel of $\e{- \beta \Delta}$ is given by $\frac{1}{(4 \beta)^{d/2}} \e{-\frac{1}{4 \beta} |x-y|^2}$. 
This regularly leads to a lot of confusion, which is why we avoid using the term 'inverse temperature'. 

There are further problems 
when trying to view the measure \eqref{measure} as a finite volume Gibbs measure. 
An obvious attempt is to regard it as a system of $X$-valued spins, where the jumps 
$\eta(x) = \pi(x) - x$ take the role of the spins. This system only has a 
single site potential given by $\xi$, but the requirement that $x \mapsto \pi(x) = \eta(x) + x$ 
must be bijective on $X$ imposes a hard core condition of infinite range. 
This excludes the use of almost all the classical theory of spin systems or Gibbs measures. 
One consequence is that proving the existence of an infinite volume measure is 
more difficult than one might first think.

\section{Infinite volume limit} \label{ivl}

Here we prove the existence of a measure of spatial random permutations 
on an infinite regular lattice. Since the set $\caS(X)$ of permutations of a countable set is 
uncountable, we first need to 
define a suitable sigma-algebra on $\caS(X)$. This, along with an abstract
criterion for the existence of an infinite volume measure, has been done in \cite{BU1}. We 
review the construction and the result here.  

For $x,y \in X$, we write  
\[
B_{x,y} = \{ \pi \in \caS(X): \pi(x) = y \}.
\]
for the cylinder set of permutations sending $x$ to $y$. 
By a Cantor diagonal argument, the sequences $(\bbP_{N(L)}(B_{x,y}))_{L \in \bbN}$
converge along a joint subsequence for all 
$x,y \in X$, and this implies the existence of a limiting finitely additive set function $\mu$ 
on the semi-ring $\Sigma'$ generated by the cylinder sets.
If $\mu$ extends to a probability measure $\bbP$ on $\caS(X)$ (equipped with the sigma 
algebra generated by the cylinder sets), we will say that $\bbP$ is a measure of spatial random 
permutations on $X$ with jump energy $\xi$ and parameter $\alpha$. 
In Theorem 3.2 of \cite{BU1} it is shown that 
a necessary and sufficient condition for $\mu$ to extend to such a probability measure is that for all fixed 
$x \in X$, 
\be \label{tight}
\sum_{y \in X} \mu(B_{x,y}) = \sum_{y \in X} \mu(B_{y,x}) = 1.
\ee
For finite $A \subset X$, $\bigcup_{y \in A} B_{x,y}$ and $\bigcup_{y \in A} B_{y,x}$ are elements of 
$\Sigma'$. By finite additivity, and since the cylinder sets are mutually disjoint, \eqref{tight} follows
if for all $\delta > 0$ we can find a finite $A \subset X$ so that 
\be \label{tight2}
\lim_{L \to \infty} \bbP_N \Big( \bigcup_{y \in A} B_{x,y} \Big) \geq 1-\delta, \qquad 
\lim_{L \to \infty} \bbP_N \Big( \bigcup_{y \in A} B_{y,x} \Big) \geq 1-\delta
\ee
for the converging subsequence. In words, both the image and the pre-image of each point $x$ need to 
stay within a finite set with high probability, uniformly in $N$. For a system of independent $X$-valued 
spins with single site potential $\xi$, the tightness criterion \eqref{tight2} would hold  
as soon as $x \mapsto \e{-\xi(x)}$ is summable over $x$. The following theorem shows that for 
SRP on a periodic lattice, the condition is only marginally stronger. 

\begin{theorem} \label{existence}
Let $X_L$ be a periodic regular lattice of 
side length $L$, and  define $\bbP_N = \bbP_{N(L)}$ as in \eqref{measure}. For 
$\alpha >0$, assume that there exists $\eps > 0$ such that 
\be \label{existence condition}
\sum_{x \in X} \e{- (\alpha - \eps) \xi(x)} < \infty.
\ee
Then for all $x \in X$ and all $\delta > 0$, there exists $D>0$ such that 
\[
\sup_{L \in \bbN} \bbP_N (|\pi(x) - x| > D) < \delta, 
\qquad \sup_{L \in \bbN} \bbP_N (|\pi^{-1}(x) - x| > D) < \delta.
\]
In particular, an infinite volume measure of spatial random permutations with jump energy 
$\xi$ and parameter $\alpha$ exists. 
\end{theorem}

Condition \eqref{existence condition} implies that $\lim_{|x| \to \infty} \xi(x) = \infty$. It does not 
imply that the minimum of $\xi$ is at $x=0$. 
When $\xi$ grows faster than logarithmically at infinity, 
\eqref{existence condition} can be met for all $\alpha > 0$, and thus an infinite volume limit exists 
for all $\alpha$ in these cases. 

We prepare the proof by giving a couple of results that exploit the translation invariance. 

\begin{lemma} \label{equal probabilities}
Under the assumptions of Theorem \ref{existence}, for all $x,k \in X_L$ and all $m \in \bbN$ we have
\[
\bbP_N (\pi(x) - x = k) = \bbP_N(\pi^{m}(x) - \pi^{m-1}(x) = k).
\]
\end{lemma}

\begin{proof}
For $m=1$ there is nothing to prove. Let us now assume the claim holds for $m \in \bbN$. For 
$y \in X_L$, write $\sigma_y$ for the permutation of $X_L$ with $\sigma_y(x) = x+y$; the sum is understood with 
the periodic boundary conditions explained above. Define the map $T_y$ on $\caS(X_L)$ by 
$T_y \pi = \sigma_y \pi \sigma_{-y}$. Since
$\sigma_y^{-1} = \sigma_{-y}$ we have $T_y \pi (x)  = \pi(x-y) + y$, i.e. $T_y \pi$ 
is the permutation where all jumps of $\pi$ are shifted by $y$. Thus $T_y$ is a bijection on 
$\caS(X_L)$ with $T_y^{-1} = T_{-y}$, and $H_N(T_y \pi) = H_N(\pi)$. This implies 
\[
\bbE_N(F \circ T_y) = \frac{1}{Z_N} \sum_{\pi \in \caS_N} \e{-\beta H(T_{y} \pi)} F(T_y \pi) 
= \bbE_N(F)
\]
for all bounded random variables $F$. We apply this to 
$F(\pi) = 1_{\{\pi^m(x) - \pi^{m-1}(x) = k\} } 1_{\{ \pi^{-1}(x)=y \} }$ and use 
the fact that $(T_y \pi)^{m} = T_y \pi^{m}$ for all $m \in \bbZ$.  
We thus find 
\[
\begin{split}
\bbP_N(\pi^{m}(x) - \pi^{m-1}(x) = k) & = \sum_{y \in X_L} \bbP_N \Big( \pi^{m}(x)-\pi^{m-1}(x) = k, \pi^{-1}(x) = x+y \Big) \\
& = \sum_{y \in X_L} \bbP_N \Big( T_y \pi^{m} (x) - T_y \pi^{m-1}(x) = k, T_y \pi^{-1}(x) = x+y \Big) \\
& = \sum_{y \in X_L} \bbP_N \Big( \pi^m(x-y) - \pi^{m-1}(x-y) = k, \pi^{-1}(x-y) = x \Big) \\
& = \sum_{y \in X_L} \bbP_N \Big( \pi^{m}(\pi(x)) - \pi^{m-1}(\pi(x)) = k, \pi(x) = x-y \Big) = \\
& = \bbP_N(\pi^{m+1}(x) - \pi^{m}(x) = k) .
\end{split}
\]
\end{proof}

For the next statement, let $C_x(\pi) = \{ \pi^n(x): n \leq N \}$ 
denote the orbit of $x$ under $\pi$, i.e.\ the cycle containing $x$. We write  
\[
R_{x,D}(\pi) = \frac{1}{|C_x(\pi)|} \Big | \{ y \in C_x(\pi): |\pi(y)-y| > D \} \Big|
\]
for the fraction of jumps of length greater than $D$ in the cycle containing $x$. Above and 
in future, $|A|$ denotes the cardinality of a finite set $A$. 

\begin{lemma} \label{prob is exp} Under the assumptions of Theorem \ref{existence}, we have 
\[
\bbP_N( |\pi(x) - x| > D) = \bbE_N (R_{x,D}) \qquad \text{ for all } x \in X, D > 0.
\] 
\end{lemma}

\begin{proof}
Since $\pi^{|C_x(\pi)|}(x)=x$, clearly 
\[
R_{x,D}(\pi) = \frac{1}{n |C_x(\pi)|} \sum_{j=1}^{n |C_x(\pi)|} 1_{\{|\pi^{j}(x) - \pi^{j-1}(x)| > D\}}
\]
for all $n \in \bbN$. By taking $M$ to be the smallest common multiple of $\{1, \ldots, N\}$
we can make the denominator and the number of terms non-random, and with 
Lemma \ref{equal probabilities}, we find
\[
\bbE_N(R_{x,D}) = \frac{1}{M} \sum_{j=1}^{M} \bbP_N( |\pi^{j}(x) - \pi^{j-1}(x)| > D ) = \bbP_N(|\pi(x)-x|>D).
\]
\end{proof}

\begin{proof}[Proof of Theorem \ref{existence}]
The idea is that for fixed $x \in X$, either long jumps are rare in $C_x(\pi)$, in which case 
they are unlikely to occur precisely at the particular $x$ under consideration;
or, long jumps are common, in which case a variant of the 'high temperature' 
argument given in the proof of 
Theorem 4.1 of \cite{BU1} will ensure that such permutations are unlikely. 

Fix $\delta > 0$. For each $D>0$, Lemma \ref{equal probabilities} and Lemma \ref{prob is exp} imply  
\be \label{the split}
\bbP_N(|\pi(x)-x| > D) =\bbP_N(|\pi^{-1}(x)-x| > D) = \bbE_N(R_{x,D}) 
\leq \delta + \bbE_N( R_{x,D}, R_{x,D} > \delta).
\ee
Above, we used the notation $\bbE(f, A)$ 
instead of $\bbE(f 1_A)$ for random variables $f$ and measurable sets $A$, 
and will continue to do so. 
We will show that under condition \eqref{existence condition}, there exists $D>0$ such that 
\be \label{toProve}
\sup_{L} \bbE_N \Big( \e{R_{x,D}}, R_{x,D} > \delta \Big) \leq \delta. 
\ee
Using the inequality $\e{x} > x$ we then find that $\bbE_N( R_{x,D}, R_{x,D} > \delta) \leq \delta$ 
uniformly in $N$, and the claim follows. 

Turning to the proof of \eqref{toProve}, we start by adding a constant to $\xi$ so that 
$\xi(0) = 0$. This does not change the measure $\bbP_N$, but 
$\xi_- = \min_{x \in X} \xi(x)$ may be strictly negative. By \eqref{existence condition}, for all $C>0$ 
we can find $D>0$ such that $\xi(x) > C$ whenever $|x| > D$. Therefore, 
\[
\begin{split}
R_{x,D} & = \frac{1}{|C_x(\pi)|} \sum_{j=1}^{|C_x(\pi)|} 1_{\{ |\pi^j(x) - \pi^{j-1}(x)| > D \} } \\
& \leq \frac{1}{C |C_x(\pi)|} \sum_{j=1}^{|C_x(\pi)|} \xi( \pi^j(x) - \pi^{j-1}(x) ) 
1_{\{ |\pi^j(x) - \pi^{j-1}(x)| > D \} } \\
& \leq \frac{1}{C |C_x(\pi)|} \sum_{j=1}^{|C_x(\pi)|} \Big( \xi( \pi^j(x) - \pi^{j-1}(x) ) - \xi_- \Big).
\end{split}
\]
Therefore, 
\be \label{eq1}
\bbE_N \Big( \e{R_{x,D}}, R_{x,D} > \delta \Big) \leq \sum_{k=2}^N 
\sum_{j= \lceil \delta k \rceil}^k \bbE_N \Big( \e{\tfrac{1}{kC} 
\sum_{y \in C_x(\pi)} \xi(\pi(y)-y) - \frac{\xi_-}{C}}, 
|C_x| = k, \, R_{x,D} = \frac{j}{k} \Big).
\ee
Let us write $M_{x,k,j}$ for the set of all vectors $(x_0,x_1, \ldots, x_{k-1}) \in (X_L)^k$ for which 
$x_0 = x$,  $x_i \neq x_j$ if $i \neq j$, and $|x_i - x_{i-1}| > D$ precisely $j$ times. 
For $\bsx \in M_{x,k,j}$, write 
\[
\caS_{N,\bsx} := \{ \pi \in \caS(X_L): |C_x(\pi)| = k, \pi^i(x) = x_i \text{ for all } 0 \leq i < k \}.
\]
With the convention $x_k = x_0 = x$, for fixed $k$ and $j$ the expectation 
on right hand side of \eqref{eq1} then equals 
\be \label{eq2}
\frac{1}{Z_N} \sum_{\bsx \in M_{x,k,j}} \e{- (\alpha - \tfrac{1}{kC}) \sum_{i=1}^k \xi(x_i - x_{i-1}) 
- \frac{\xi_-}{C}}
\sum_{\pi \in \caS_{N,\bsx}} \e{-\alpha \sum_{y \notin \{x_1, \ldots, x_k \} 
} \xi(\pi(y)-y)} = (\ast).
\ee

By our normalization $\xi(0)=0$, we have 
\[
\frac{1}{Z_N} \sum_{\pi \in \caS_{N,\bsx}} \e{-\alpha \sum_{y \notin \{x_1, \ldots, x_k \} 
} \xi(\pi(y)-y)} = \bbP_N \Big( \pi(y) = y \text{ for all } y \in \{x_1, \ldots, x_n \} \Big) \leq 1.
\]
Now we use \eqref{existence condition} and choose $\eps>0$ such that 
\[
Q := \sum_{x \in X} \e{- (\alpha - \eps) \xi(x)} < \infty.
\]
By taking $D$ large enough, we can make 
$\frac{1}{2C} < \eps$. In the first sum of \eqref{eq2}, we now 
relax the condition that $x_i \neq x_j$ for $i \neq j$, 
allow the last jump to go anywhere instead of back to $x$, and allow jumps to go to any point of $X$ 
instead of just into $X_N$. Writing $\rho = \alpha - \eps$, we then find that 
\be \label{ast}
(\ast) \leq \e{-\frac{\xi_-}{C}} \sum_{x_1 \in X} \e{- \rho \xi(x_1-x)} 
\ldots \sum_{x_{k} \in X} \e{- \rho \xi(x_{k}-x_{k-1})} 1_{\{|x_i-x_{i-1}| > D \text{ holds $j$ times} \}}.
\ee
In the above expression, $j$ of the sums above are over $\{ x_i \in X: |x_i - x_{i-1}| > D \}$, 
and thus bounded by 
\[
\sup_{y \in X} \sum_{x_i: |x_j - y| > D} \e{-\rho \xi(x_j-y)} = \sum_{z \in X: |z| > D} \e{-\rho \xi(z)} =: 
\eps_D.
\]
By choosing $D$ even larger, we can 
make $\eps_D$ as small as we want.  
The remainder of the terms in \eqref{ast} are bounded by 
$Q$. 

Taking the different orders in which large and small 
terms can appear into account, we find 
\[
\bbE_N \Big( \e{\tfrac{1}{kC} \sum_{y \in C_x(\pi)} \xi(\pi(y)-y)}, 
|C_x| = j, \, R_{x,D} = \frac{j}{k} \Big) \leq \binom{k}{j} \eps_D^{j} Q^{k-j}.
\]
For positive $\mu<1/2$, we now choose $D$ so large that $\eps_D < (Q^{\delta-1} \mu)^{1/\delta}$.
This achieves 
$\eps_D^{j} Q^{k-j} < \mu^k$ for all $j > \delta k$, and thus 
\[
\e{\xi_-} \bbE_N \Big( \e{R_{x,D}}, R_{x,D} > \delta \Big) \leq \sum_{k=2}^N 
\sum_{j= \lceil \delta k \rceil}^k \binom{k}{j} \eps_D^j Q^{k-j} \leq \sum_{k=1}^N 
\sum_{j= 0}^k \binom{k}{j} \mu^k 
\leq \frac{2 \mu}{1 - 2 \mu}.
\]
Taking $D$ so large that  
$\mu < \frac{1}{2} \tfrac{\delta \e{\xi_-}} {1 + \e{\xi_-} \delta}$, the last expression is
less than $\delta$, and this finishes the proof. 
\end{proof}

We have used a compactness argument for showing the existence of an infinite 
volume limit. This leaves open the question of uniqueness, and of a possible 
DLR characterization. Both of these questions are addressed 
by Biskup and Richthammer in \cite{BiRi_unplublished}, which is at present the only result on 
the existence of infinite volume SRP measures besides the one given above. 
Biskup and Richthammer treat the case where $X$ is a locally finite subset 
of the real line fulfilling weak regularity conditions, 
and where $\xi$ satisfies a strong form of convexity. 
They show that under these circumstances, 
infinitely many infinite volume Gibbs measures exist, fulfill a set of DLR 
conditions, and are characterized by the 'flow' across $x=0$, 
i.e.\ net number of jumps passing across the origin in the positive direction. 
It is likely that by combining the methods presented above 
with those of \cite{BiRi_unplublished}, one could prove similar results for SRP where 
$X=\bbZ$ and the potential $\xi$ fulfills only \eqref{existence condition}.
In higher dimensions, the results of \cite{BiRi_unplublished} suggest that there
will always be at least countably many extremal Gibbs measures, characterized by 
the vector of winding numbers around the infinite torus. The latter quantity would be 
the $L \to \infty$ limit of \eqref{winding vector} below, assuming it exists. We expect the
proof of such a statement to be difficult. 

\section{Markov chain Monte Carlo} \label{mcmc}

The remainder of this paper deals with numerical results about long cycles in planar spatial 
random permutations. From now on, we will always assume either $X = \bbZ^2$, or the triangular lattice, 
where the fundamental region is an equilateral triangle of side length $(4/3)^{1/4}$. The side length 
is chosen such that the large scale density of points stays equal to one, which is important as we want
to quantitatively compare results for the two lattices. We will also restrict our attention to quadratic jump 
energy. Therefore, \eqref{energy} will now always read 
\be \label{energy2}
H_N(\pi) = \sum_{x \in X_L} |x - \pi(x)|^2,
\ee
with $|.|$ the periodic distance on $X_L$.

We use the Metropolis algorithm described in \cite{GLU12}: 
we choose a pair $(x,y)$ of points uniformly from all pairs of nearest neighbors in $X_L$
and propose to exchange their targets. In other words, 
given a pair $(x,y)$ and a permutation $\pi$, let $\pi'$ be the permutation such that $\pi'(z) = \pi(z)$ for 
all $z \notin \{x,y\}$, $\pi'(y) = \pi(x)$ and $\pi'(x) = \pi(y)$. The energy difference between these 
permutations is 
\[
\Delta H(\pi,\pi') = H(\pi')-H(\pi) = |\pi(y) - x|^2 + |\pi(x) - y|^2 - |\pi(x) - x|^2 - 
|\pi(y) - y|^2.
\]
A {\em Metropolis step} consists in choosing the pair $(x,y)$ uniformly at random, and 
switching from $\pi$ to $\pi'$ with probability $\e{- \alpha \Delta H(\pi,\pi')}$ 
if $\Delta H(\pi,\pi') >0$, and with probability one otherwise. It is not difficult to see, and was shown in \cite{Ke10},
that the resulting Markov chain is ergodic, and that the transition 
rates satisfy the detailed balance condition with respect to the measure \eqref{measure}. A 
{\em Metropolis sweep} consists of $N = |X_L|$ separate Metropolis steps; thus on average, each point in the 
lattice had the chance to exchange its target twice during each sweep. As observed in \cite{GLU12}, the 
dynamics of the Metropolis chain effectively loses ergodicity in the limit of large systems. 
The reason is that the winding number of a permutation is strongly metastable. More precisely, if we define
the vector of winding numbers
\be \label{winding vector}
w(\pi) := \frac{1}{L} \sum_{x \in X_N} (\pi(x) - x),
\ee
then each cycle of $\pi$ that winds around the 
torus $X_L$ in a given coordinate direction contributes $1$ or $-1$ to the corresponding component of $w$. 
For the vector of absolute winding numbers 
\[
|w|(\pi) = \frac{1}{L} \sum_{C \text{ cycle of } \pi} \Big| \sum_{x \in C} (\pi(x)-x) \Big|,
\]
all winding cycles contribute $1$. Both quantities are strongly metastable, 
since to dynamically destroy a winding 
cycle we need jumps of length $L/2$ to appear, which is exponentially unlikely. This poses a serious 
problem when one tries to numerically simulate winding numbers, but is irrelevant for our purposes. 
On the contrary, we will be able to use metastability to our advantage when computing 
a numerical approximation to the fractal dimension of long cycles for large $\alpha$. 

We do, however, use the swap-and-reverse 
algorithm given in \cite{Ke10}. This means that some of the longer cycles (say, the ten longest) are reversed, with 
probability $1/2$ each, after a certain number of sweeps, i.e.\ each jump from $x$ to $\pi(x)$ is replaced by 
one from $\pi(x)$ to $x$. One effect of this procedure is that the winding number can now change dynamically by 
multiples of two. More importantly, we have found that reversing cycles after a small number of sweeps 
improves the mixing properties of the dynamics, even though identifying the longest cycles in a 
permutation is computationally somewhat expensive. 

We have simulated system sizes between $L = 1000$ and $L = 4000$, 
implying $|X_L|$ between $10^6$  and $1.6 * 10^7$. A difficulty with planar permutations that seems to be 
absent in higher dimensions is that a very long thermalization time is necessary for some observables. 
When starting the MCMC chain from the 
identity permutation, the average jump length $\frac{1}{N} \sum_{x \in X_L} |\pi(x) - x|$ is well 
converged to its long time average after around $100$ sweeps. This is comparable to the thermalization time 
observed in \cite{GLU12} for the fraction of points in macroscopic cycles. However, quantities that depend
on the occurrence of large cycles take much longer to thermalize in the two-dimensional case. 

\bfig
\centerline{{\tiny a)}
\includegraphics[width=0.47 \textwidth]{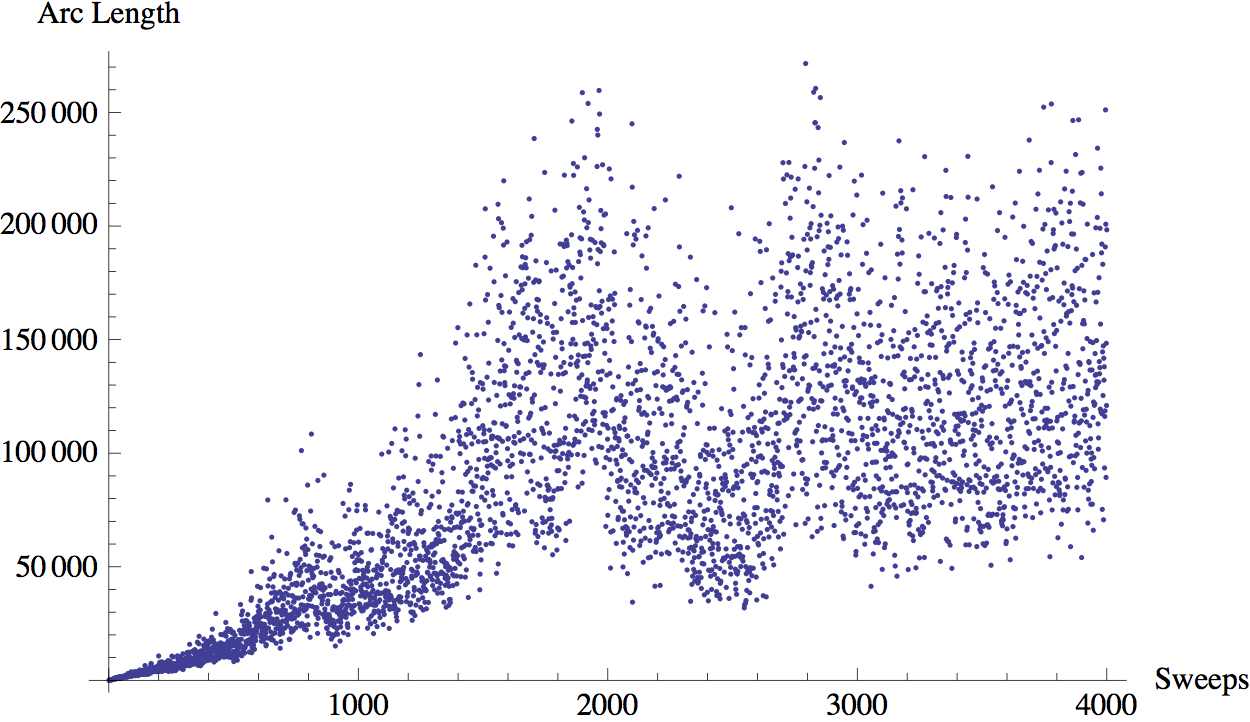} 
{\tiny b)}\includegraphics[width=0.47 \textwidth ]{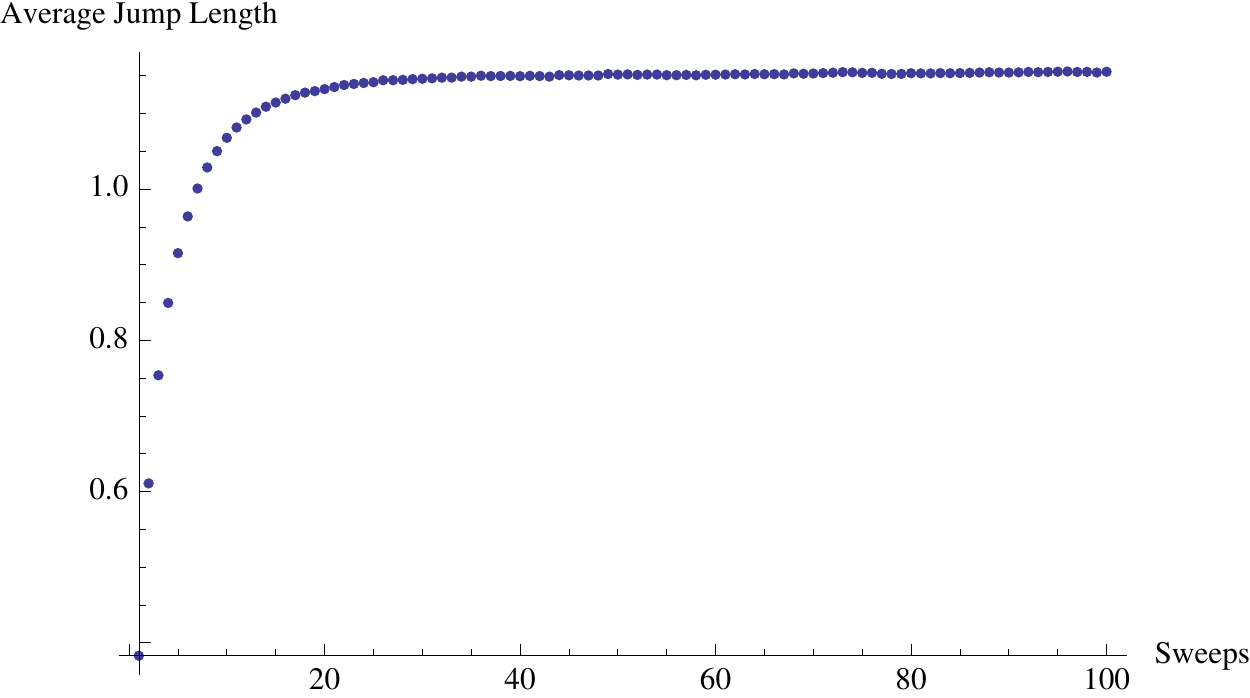}}
\caption{ a) The arc length of the longest cycle during the first 4000 thermalization sweeps on a $1000 \times 1000$ lattice, for $\alpha=0.5$. 
b) The average cycle length during the first $100$ thermalization steps of the system described in a).}
\label{fig2}
\efig

As an illustration, Figure \ref{fig2} a) shows the arc length (as a piecewise constant curve in $\bbR^2$) 
of the longest cycle after each sweep, for the first $4000$ sweeps of a random permutation 
on a square grid of side length $L=1000$, at $\alpha = 0.5$.
The initial condition is the identity permutation. We see that it takes about $2000$ sweeps to get even close 
to equilibrium. After that, the maximal arc length has very large fluctuations; this is expected 
since a 
single Monte Carlo step can break a long cycle into two roughly equally large pieces 
(see \cite{GLU12} for more information on the corresponding split-merge mechanism). Some sort of 
equilibrium is reached eventually, 
but we note that the maximal lengths at around 2500 sweeps are significantly 
smaller than is typical. These strong time correlations seem not to be connected with thermalization; 
we have observed them even after millions of sweeps. 
In contrast, Figure \ref{fig2}b) shows the average jump length during the first 100 
sweeps of the same simulation that 
was used for Figure \ref{fig2}a). Clearly, the average jump length thermalizes on a different scale.

As a consequence, we have to be careful when simulating planar SRP. In all of the simulations below, we 
thermalize for at least $100000$ sweeps, take many hundreds of samples, 
and allow at least $10$ sweeps in between consecutive samples. 

\section{Kosterlitz-Thouless phase transition} \label{section KT}

In three-dimensional SRP, the existence of a phase transition 
is known in the annealed case \cite{BU1, BU4}, and numerically evident in the lattice case 
\cite{GRU}. The relevant order parameter is the fraction of points in macroscopic cycles. 
Let us give a little detail. We write 
\be \label{ell_x}
\ell_x(\pi) = \inf \{ n \in \bbN: \pi^n(x) = x \} \leq \infty
\ee
for the length of the cycle of $\pi$ containing $x$. 
Let $\bbP$ denote the limit of a sequence of finite volume SRP measures $\bbP_N$ along a subsequence;
we will suppress the subsequence from the notation and write $\bbP_N \to \bbP$. 
By Theorem \ref{existence}, for each $K>0$ and 
all permutations outside a set of arbitrarily small probability, 
it can be decided whether $\ell_x(\pi) > K$ by looking at the images $\pi(y)$ for all $y$ 
from a sufficiently large finite set. Thus we know that  
\be \label{fraction macro}
\nu(K) := \lim_{N \to \infty} \bbP_N(\ell_x > K) = \lim_{N \to \infty} \frac{1}{N} \bbE_N( \sum_{x \in X_L} 
1_{\{\ell_x(\pi) > K\}})
\ee
exists. $\nu(K)$ is the probability that $x$ is in a cycle longer than $K$, and is independent of $x$ by 
translation invariance.  We define  
\[
\alpha_c = \inf\{ \alpha \geq 0: \lim_{K \to \infty} \nu(K) = 0 \}
\]
as the critical parameter value where infinite cycles start to appear. 
Since the case $\alpha=0$ corresponds to uniform permutations, we know $\alpha_c \geq 0$. 
In Section 4 of \cite{BU1} it is shown that $\alpha_c < \infty$ in rather great generality; in 
particular, this includes the case considered in Theorem \ref{existence}. In the same paper, 
it is proved for the {\em annealed} case (i.e. where the measure is averaged over 
the points of the box $\Lambda$) that $\alpha_c > 0$ 
if $d \geq 3$, and  $\alpha_c = 0$ for $d \leq 2$. 
For the lattice case, no rigorous lower bounds on $\alpha_c$ exist. It is
observed numerically in \cite{GRU} that $\alpha_c >0$ if and only if $d \geq 3$. 

This does not imply that no long cycles exist in two dimensions, 
only that they are too small to hit a given point $x$ with positive probability 
in the infinite volume limit.  
Indeed, we can expect the length of the longest cycle to always diverge as $L \to \infty$, 
as the following heuristic argument shows. Fix $K \in \bbN$. 
$\eps_{K,N} :=\bbP_N(\ell_x > K)$ can be expected to be strictly positive, uniformly in $N$, 
for all fixed $K \geq 0$ and $\alpha > 0$.
For $x,y \in X$ with $|x-y| \gg K$, the events $\{\ell_x > K\}$ and $\{ \ell_y > K\}$ should 
be almost independent for sufficiently large $\alpha$. By taking $M$ points 
$x_1, \ldots, x_{M}$ far enough apart, we get  
\[
\bbP_N (\sup_x \ell_x \leq K ) \leq \bbP_N (\max_{i \leq M} \ell_{x_i} \leq K) 
\approx \bbP_N(\ell_x \leq K)^{M} = (1-\eps_K)^M
\]
for all $M>0$ and sufficiently large $\alpha$. We conclude that 
$\bbP(\sup_x \ell_x > K) =1$ for all $K$ in this case. Since it is reasonable to assume that 
$\alpha \mapsto \bbP (\sup_x \ell_x \leq K )$ is decreasing in $\alpha$, 
infinite cycles actually exist in infinite volume for all $\alpha$. Thus
neither $\bbP(\ell_x > K)$ nor $\bbP(\sup_x \ell_x > K)$ 
have interesting $K \to \infty$ limits for $d=2$.  

A quantity that does have an interesting behavior in both $d=2$ and $d=3$ is 
the rate of decay of $\bbP(\ell_x > K)$ as  
as $K \to \infty$. In Section 4 of \cite{BU1} it is shown that this decay is exponential
for sufficiently large $\alpha$. For $d=3$ 
no decay is expected for sufficiently small $\alpha$. In $d=2$, what changes with $\alpha$ 
is the rate of decay: we will give numerical evidence for a 
critical value $\alpha_c >0$ where for $\alpha > \alpha_c$, $\bbP(\ell_x > K)$ 
decays exponentially in $K$, and for $\alpha < \alpha_c$ the decay is algebraic.
This phenomenon is well known in two-dimensional systems and is called 
Kosterlitz-Thouless phase transition.

The best known example of a model exhibiting the Kosterlitz-Thouless phase transition 
is the $XY$-model, also known as plane rotor model. 
It is a lattice spin system, where a $2$-dimensional unit vector $v_x$ 
is attached to each point of $\bbZ^2$, and the Hamiltonian in a finite 
volume $\Lambda \subset \bbZ^2$ 
is given by $H(\bsv) = - \sum_{x \sim y} v_x \cdot v_y$. The sum is meant to be over 
pairs of nearest neighbors in $\Lambda$. 
The finite volume Gibbs measure at 
inverse temperature $\beta$ is defined in the usual way by 
\[
\mu_\Lambda(\dd \bsv)  = \frac{1}{Z_\Lambda} \e{- \beta H(\bsv)} \dd \bsv,
\]
with $\dd \bsv$ denoting the $|\Lambda|$-fold product of the uniform measure on the unit circle.
The measure is 
invariant under simultaneous rotation of all the vectors by the same angle, and by the 
Mermin-Wagner theorem \cite{MW66} this continuous symmetry cannot be broken 
in the infinite volume limit.
Thus, even though the interaction encourages neighboring vectors to point in similar directions, 
in a two-dimensional infinite volume measure 
the spins $v_j$ do not have a (random) preferred direction, i.e.\
$\lim_{|\Lambda| \to \infty} |\bbE_{\Lambda}(v_x)| = 0.$
In \cite{FS81}, Fr\"ohlich and Spencer prove the presence of a Kosterlitz-Thouless 
phase transition in the plane rotor model and several other two-dimensional classical lattice models.
In the plane rotor model, it means that the correlation $\bbE(v_x \cdot v_y)$ decays as a power law 
in $|x-y|$ for large $\beta$, and exponentially for small $\beta$. As remarked already in the original 
article by Kosterlitz and Thouless \cite{KT73}, this is connected with the formation of vortices; 
the latter are topological defects of the system, where summing the angle of $v_x$ along a large closed loop encircling the center of a vortex
will give a multiple of $2 \pi$. It is argued in \cite{KT73} that below the critical temperature, 
vortices appear in pairs, leading to algebraic decay of correlation; while above the critical 
temperature, they unbind and isolated vortices occur, leading to exponential decay of correlations. 
Thus, the phase of algebraic decay of correlations is characterized by the {\em absence} of 
topological defects. 

\bfig
\begin{center}
{\tiny a)}\includegraphics[width=0.47 \textwidth]{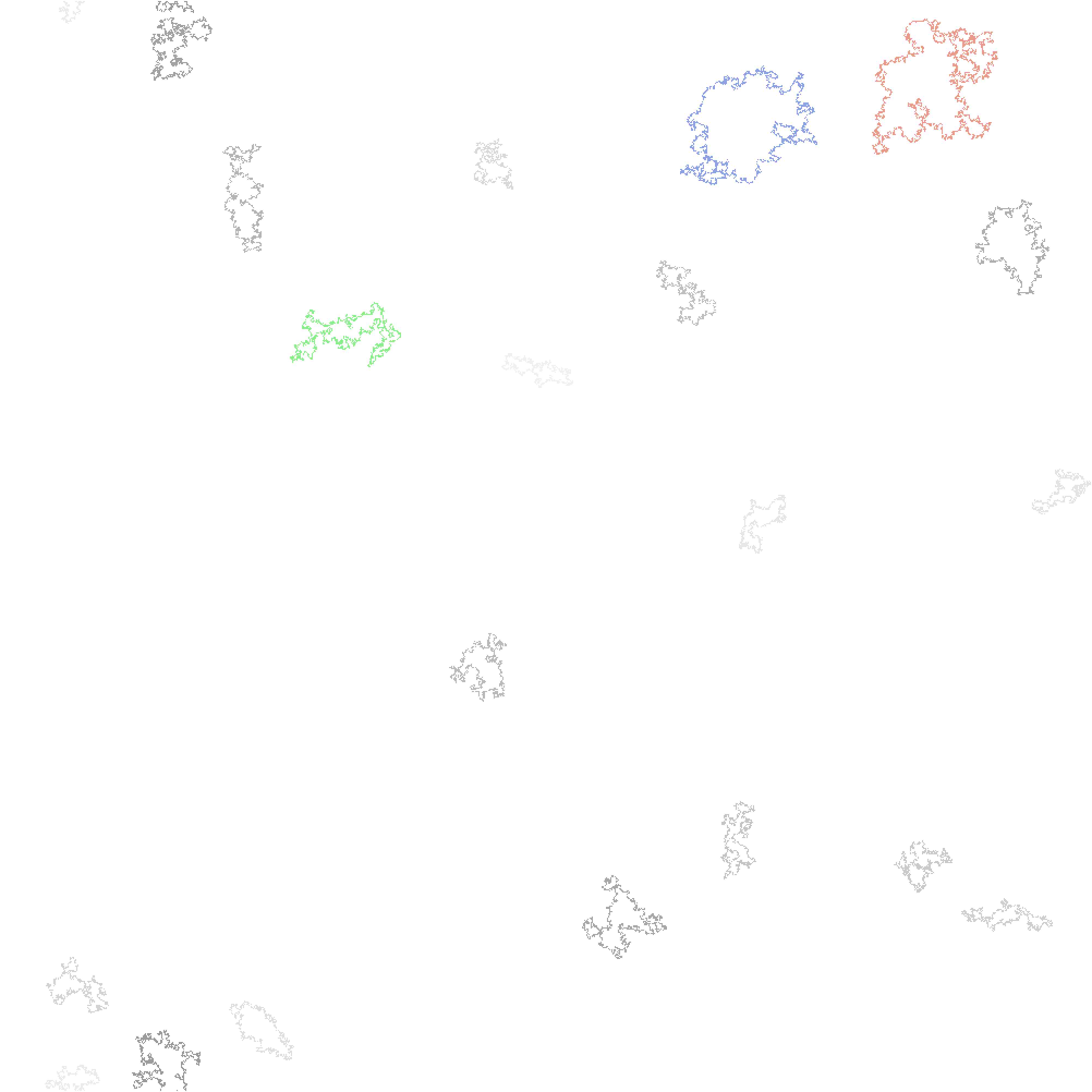} 
{\tiny b)}\includegraphics[width=0.47 \textwidth]{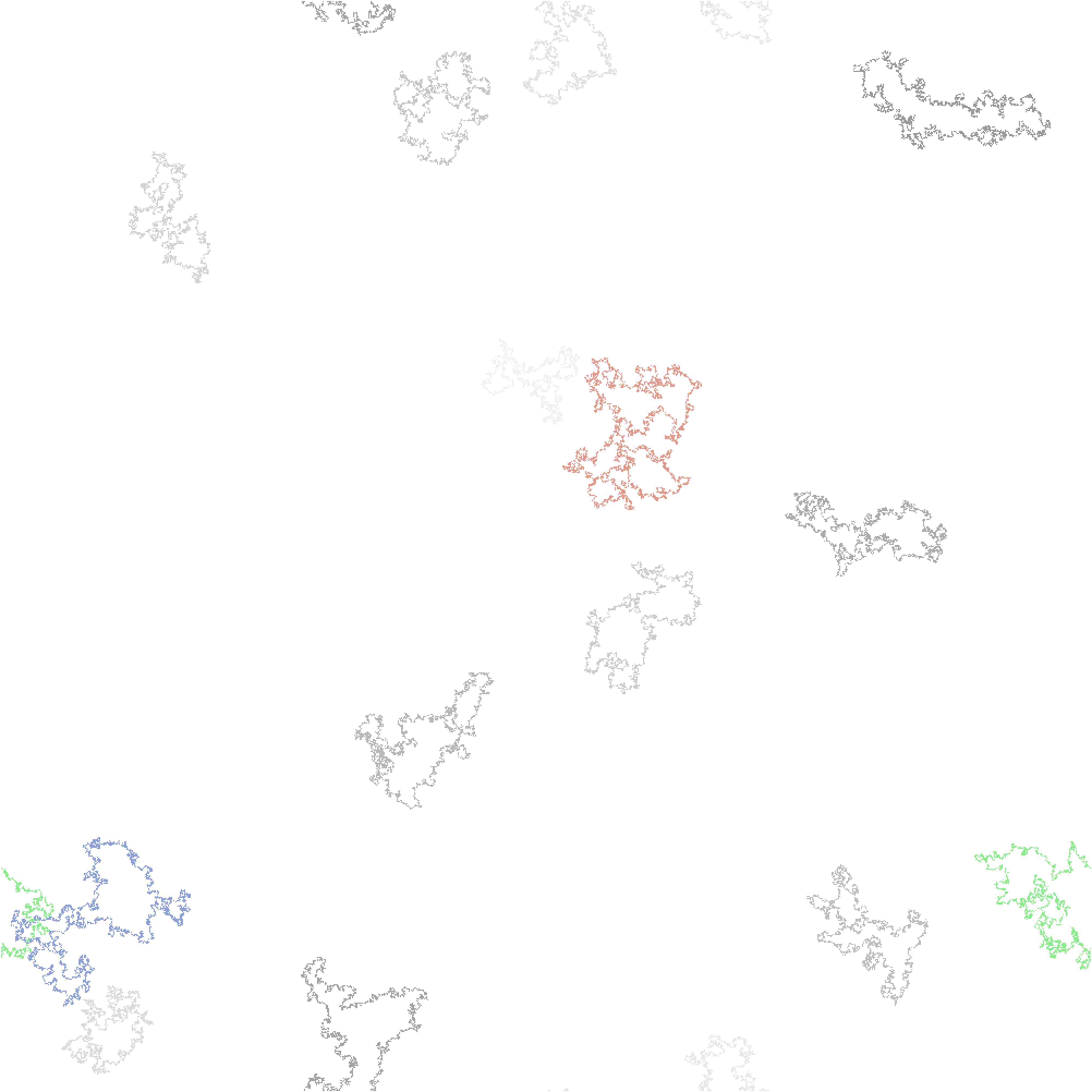}\\[4mm]
{\tiny c)}\includegraphics[width=0.47 \textwidth]{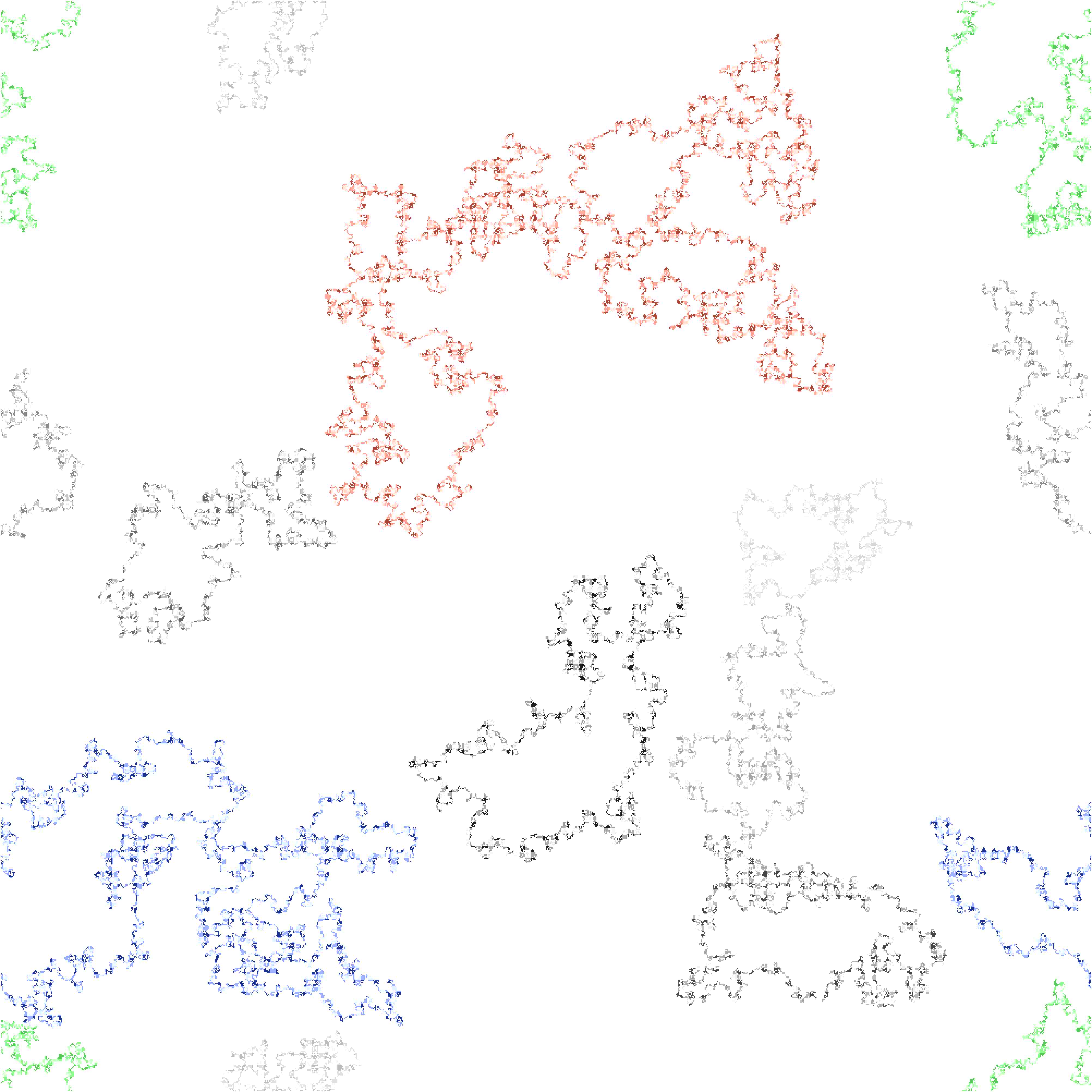} 
{\tiny d)}\includegraphics[width=0.47 \textwidth]{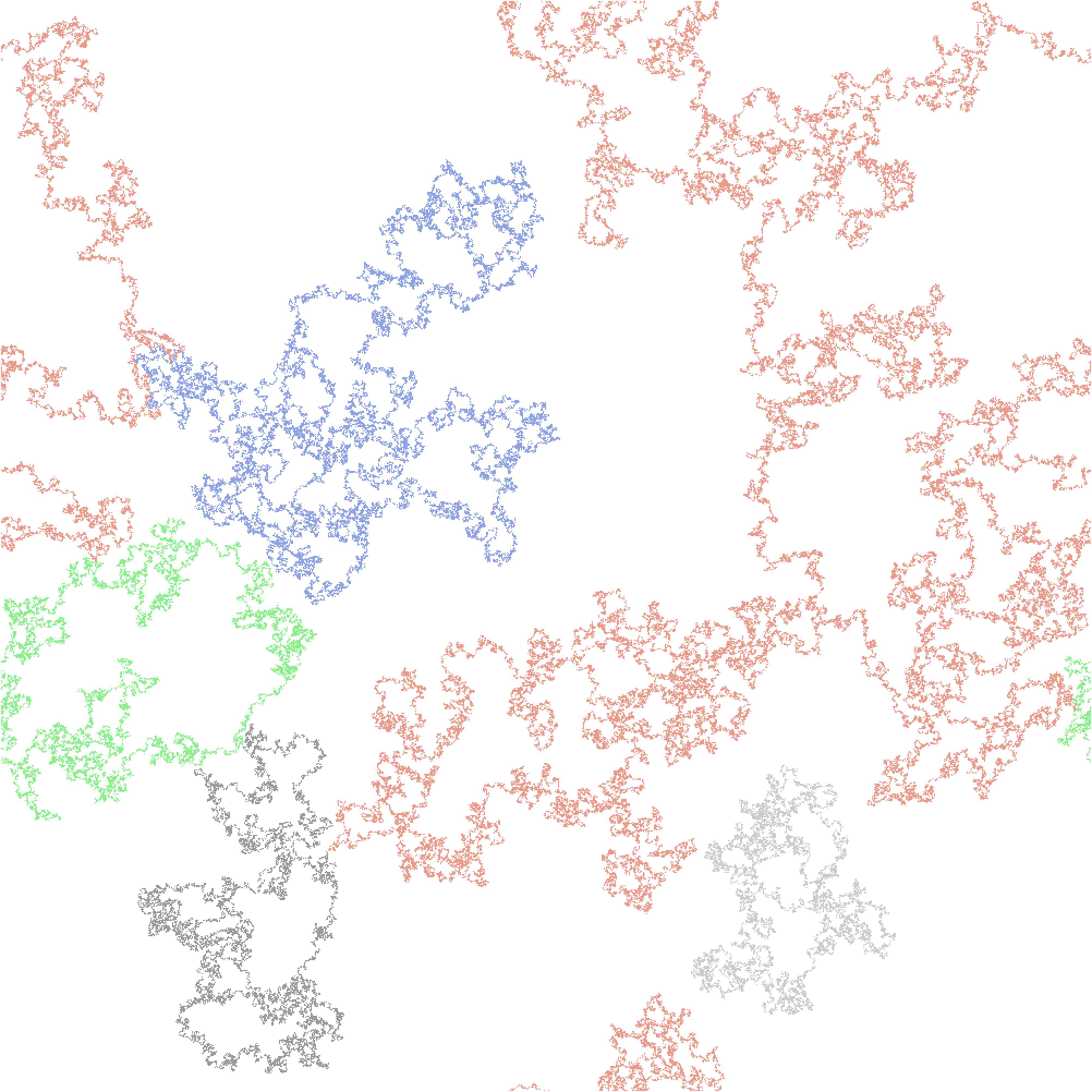}
\end{center}
\caption{The images show the lattice points belonging to 
the longest cycles of a random permutation on a square grid of side length $1000$. 
The three longest cycles are in red, blue, green, in this order. 
a) $20$ longest cycles for $\alpha = 0.8$; b) $15$ longest cycles for $\alpha = 0.75$; 
c) $10$ longest cycles for $\alpha = 0.7$; d) $5$ longest cycles for $\alpha = 0.6$.
The formation 
of 'bubbles', growing into complicated long cycles as $\alpha$ decreases, is clearly visible.}
\label{fig3}
\efig

It is curious that the situation seems to be exactly the opposite in SRP. 
Above, we introduced $\nu(K) := \bbP(\ell_x > K)$ as the quantity that will exhibit a transition 
from exponential to algebraic decay. A quantity that is more closely connected to spatial 
correlations is the probability $\mu(x,y)$ that two points $x$ and $y$ are in the same cycle. 
Since all jumps are essentially of finite length, it is clear that $\mu(x,y) \leq c \nu(K)$ when 
$|x-y| > C K$ for sufficiently large $K$. On the other hand, a cycle of length $K^2$ containing 
a point $x$ must, in two dimensions, contain at least one point of distance $K/4$ from $x$, 
and the number of such points is of the order $K$. 
So, $\mu(x,y) \geq C'\nu(K^2) / K $ for $|x-y| \leq K/4$. Thus  
$\mu(x,y)$ decays algebraically in $|x-y|$ if and only if $\nu(K)$ does, although the power law
may be different. It is clear that $\nu(K)$ 
decays more slowly in the {\em presence} of long cycles.  
Therefore slow decay of correlations is 
linked to the system being topologically non-trivial, i.e.\ having long cycles.
In the plane rotor model, in contrast, 
algebraic decay is linked to the system being 
topological trivial, i.e.\ having no isolated vortices. 
This is another reason why it would be confusing to call $\alpha$ the inverse
temperature: slow decay of correlations 
happens when $\alpha$ is small, i.e.\ when the system
is more chaotic. This is opposite to the usual situation in statistical mechanics. 
Figure \ref{fig3} gives some illustrations of the phenomenon of topological 
'bubbles' forming in the system around a certain value of $\alpha$. 

Let us now study the behavior of $\nu(K)=\bbP(\ell_x > K)$ for large $K$ and different 
$\alpha$. For large $\alpha$, it is essentially proved in \cite{BU1} that $\nu(K)$ decays 
exponentially in $K$. Since the decay rate is not explicitly stated there, we repeat 
the argument for the convenience of the reader. 

\begin{proposition} \label{large alpha}
Let $\alpha$ be large enough so that 
$s := \sum_{x \in X \setminus \{0\}} \e{- \alpha \xi(x)} < 1$. Then 
$\bbP(\ell_x > K) \leq s^K/(1-s)$. 
\end{proposition}

\begin{proof}
As in the proof of Theorem \ref{existence}, we write $C_x(\pi)$ for the 
cycle of $\pi \in \caS_N$ containing $x$. Let 
$M_{x,k}$ be the set of {\em all} vectors $(x_0, \ldots x_{k-1})$ from $X_L^k$ 
with mutually different elements and with $x_0 = x$, and for $x \in M_{x,k}$ 
let $\caS_{N,\bsx} \subset \caS(X_L)$ be the set of permutations where 
$\pi^j(x) = x_j$ for all $j < k$ and $\pi^k(x)=x$. By the same reasoning 
as in the proof of Theorem \ref{existence}, for all $N = N(L)$ we have 
\[
\begin{split}
\bbP_N (\ell_x = k)  & = \frac{1}{Z_N} \sum_{\bsx \in M_{x,k}} 
\e{- \alpha \sum_{i=1}^k \xi(x_i - x_{i-1})} \sum_{\pi \in \caS_{N, \bsx}}
\e{-\alpha \sum_{y \notin C_x} \xi(\pi(y) - y)} \\
 & \leq \sum_{\bsx \in M_{x,k}} \e{- \alpha \sum_{i=1}^k \xi(x_i - x_{i-1})} \leq \sum_{x_1 \in X \setminus \{x\}} \e{-\alpha \xi(x_1 - x)} \times \\
 & \times \sum_{x_2 \in X \setminus \{x_1\}} \e{-\alpha \xi(x_2 - x_1)} \cdots  
\sum_{x_k \in X \setminus \{x_{k-1}\} } \e{-\alpha \xi(x_k - x_{k-1})} 
 \leq s^k.
 \end{split}
\]
Thus $\bbP(\ell_x \geq K) \leq \sum_{k=K}^\infty s^k = \frac{s^K}{1-s}$.
\end{proof}

For small $\alpha$, we have to resort to numerical simulations. By translation
invariance, $\nu(K)$ is equal to the expected fraction of points in cycles longer than
$K$, a quantity that is easy to approximate numerically. In a fully thermalized 
configuration, we compute the fraction of points in cycles longer than $N^{\gamma}$ for 
all $\gamma \in \{k/100: 1 \leq k < 100\}$; we average this vector over  
1000 samples with 
10 full sweeps between consecutive samples, thus obtaining an approximation to 
$\nu(N^\gamma)$ for all such $\gamma$. Figure \ref{fig4} a) shows a double logarithmic 
plot of $\nu(N^\gamma)$ as a function of $N^\gamma$ for those $\gamma$ 
where $\nu(N^\gamma)>0$, for  $\alpha = 0.5$ and various system sizes. 
We can clearly see the power law decay of $\nu(K)$ up to a point where the system size restricts
the formation of even longer loops and finite size effects dominate the picture. This happens
already at relatively large probabilities: 
finite size effects occur at $\nu(K) \approx 0.1$ even in the largest system that we could study. 
The reason is that the power law decay comes with a small exponent: 
in the case of $\alpha = 0.5$ we find approximately $\nu(K) \sim K^{-0.189}$.

Figure \ref{fig4} b) gives evidence of a regime change. 
It shows the $\nu(K)$ for different values of $\alpha$ in steps of $0.1$. This 
leads to different decay exponents between $\alpha = 0.3$ and, apparently,  $\alpha = 0.7$; 
at $\alpha = 0.8$ the double logarithmic plot no longer has a straight piece. Indeed, a logarithmic
plot of $\nu(K)$ (not shown) reveals exponential decay.  
We thus have numerical evidence that planar SRP 
undergo a Kosterlitz-Thouless phase transition, and that the critical 
value of $\alpha$ is smaller than $0.8$. 

\bfig
\begin{center}
{\tiny a)}
\includegraphics[width=0.47 \textwidth]{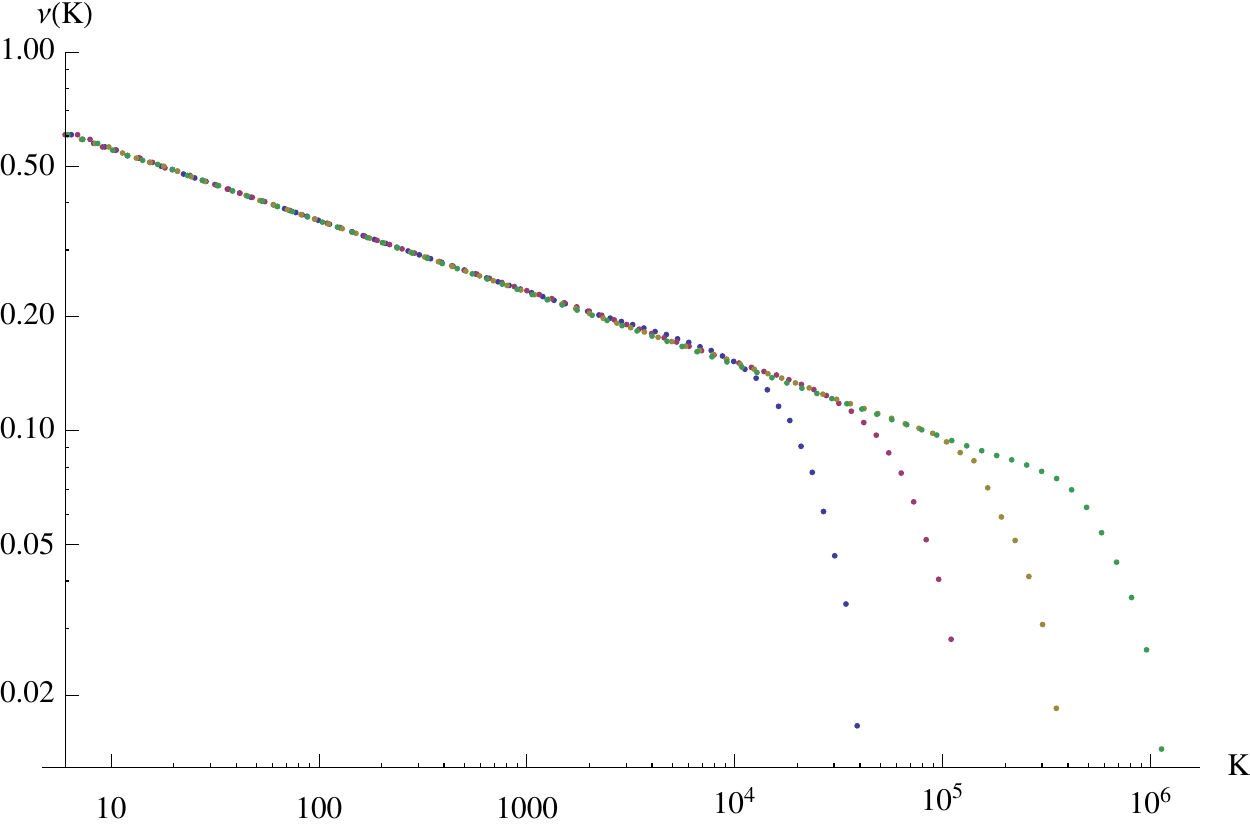} 
{\tiny b)}\includegraphics[width=0.47 \textwidth]{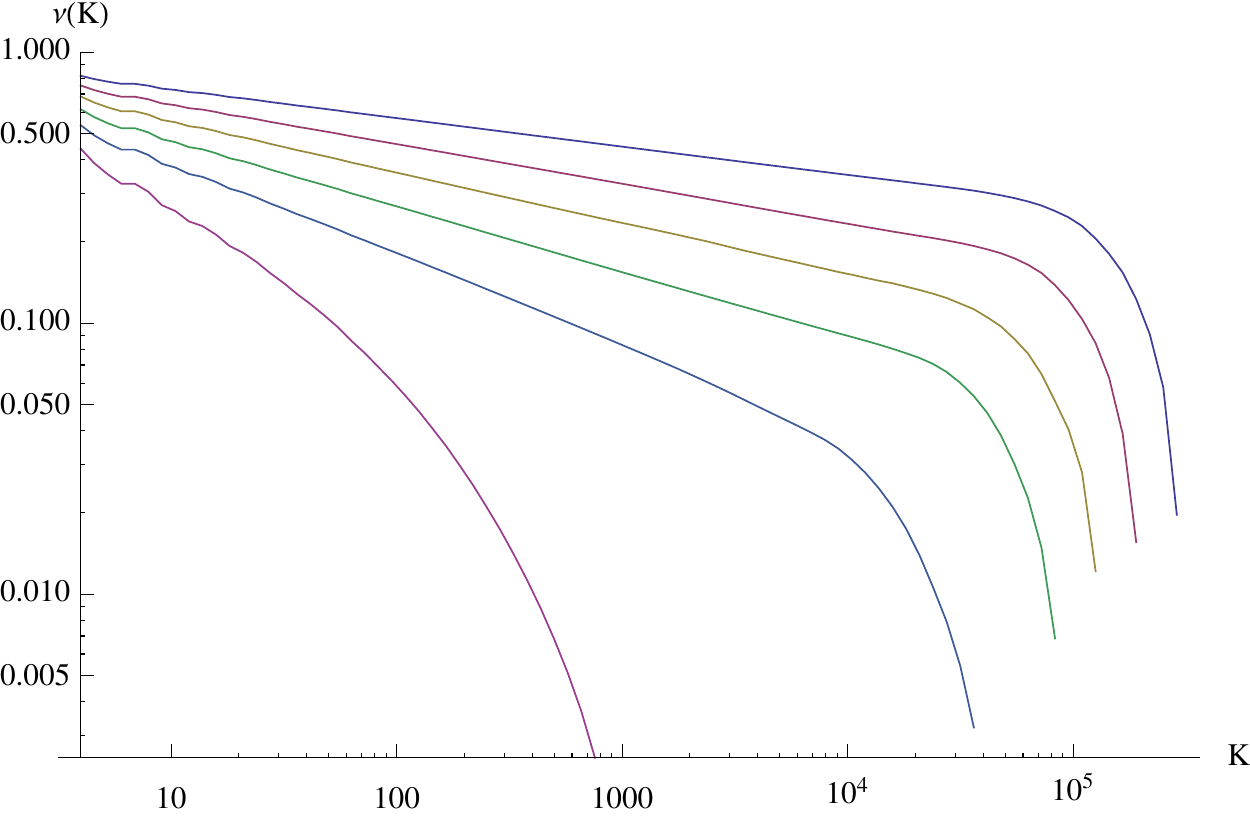}
\end{center}
\caption{The probability $\nu(K)$ that a given point is in a cycle longer than $K$. a)  
A double logarithmic plot of $\nu(K)$ at $\alpha=0.5$ and for system sizes 
$N=500^2,1000^2,2000^2,4000^2$. 
The larger system sizes maintain the straight line for larger $K$. b) $K \mapsto \nu(K)$ for 
system size $N=1000^2$, and values of $\alpha$ between $0.3$ and $0.8$, increasing in steps 
of size $0.1$. Larger values of $\alpha$ lead to faster decay of $\nu(K)$. }
\label{fig4}
\efig

To narrow down the interval in which we expect the critical value of $\alpha$, we 
investigate the decay of $K\mapsto \nu(K)$ as a function of $\alpha$. 
For small $\alpha$, the power at which $\nu(K)$ decays is the slope of the double logarithmic
plot, which was estimated by linear regression. This slope was measured for 
$\alpha$ between $0.25$ and $0.7$, increasing in steps of $0.025$. Two additional measurements 
were taken at $\alpha = 0.705$ and $\alpha = 0.71$, these being in the range in which the 
cycle length of the typical permutation changes especially rapidly, see Figure \ref{fig3}.
For even larger $\alpha$, it becomes difficult to identify a straight piece in the  
analogues of Figure \ref{fig4} b). 

The results of our simulations are shown in Figure \ref{fig5} a). The decay exponent 
$p(\alpha)$ is almost linear in $\alpha$ for $\alpha \leq 0.5$, with 
some curvature developing for larger $\alpha$. A reasonable guess for fitting a curve $p(\alpha)$ 
to the measured points is 
\be \label{formOfP}
p(\alpha) = a + b |\alpha-\alpha_0|^{\gamma},
\ee
for $\alpha < \alpha_0$; $\alpha_0$ would then be a natural candidate 
for the critical value of $\alpha$. Fitting the parameters yields 
$a \approx 0.363$, $b \approx - 0.439$, $\alpha_0 \approx 0.721$ and 
$\gamma \approx 0.617$. The curve $p(\alpha)$ for these parameters is 
also shown in Figure \ref{fig5} a), and it is a remarkably good fit. 
Although the numerical values of the constants are certainly debatable, 
one might be tempted to conclude that $p(\alpha)$ has the form given in \eqref{formOfP} with nontrivial exponent,
and that $\alpha_0$ is greater than $0.7$. 

\bfig
\begin{center}
{\tiny a)}
\includegraphics[width=0.45 \textwidth]{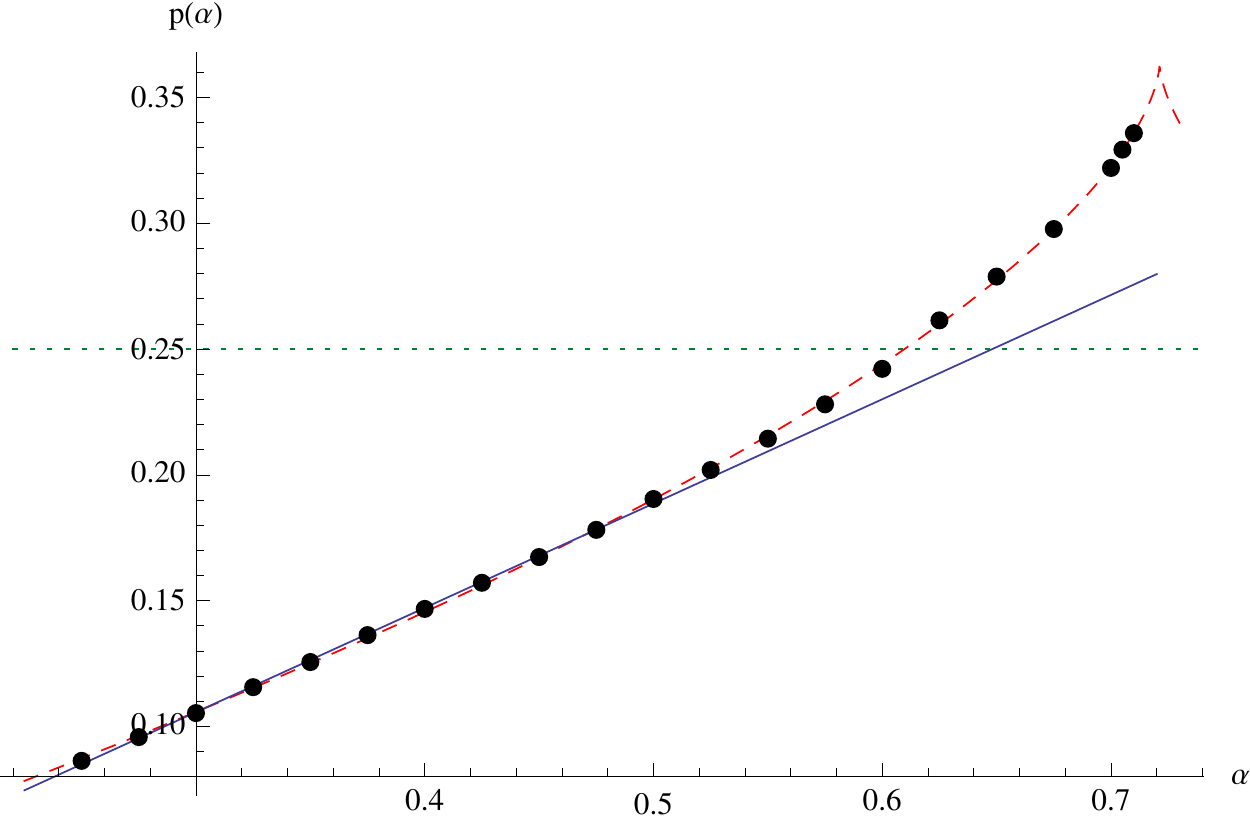} 
{\tiny b)}
\includegraphics[width=0.45 \textwidth]{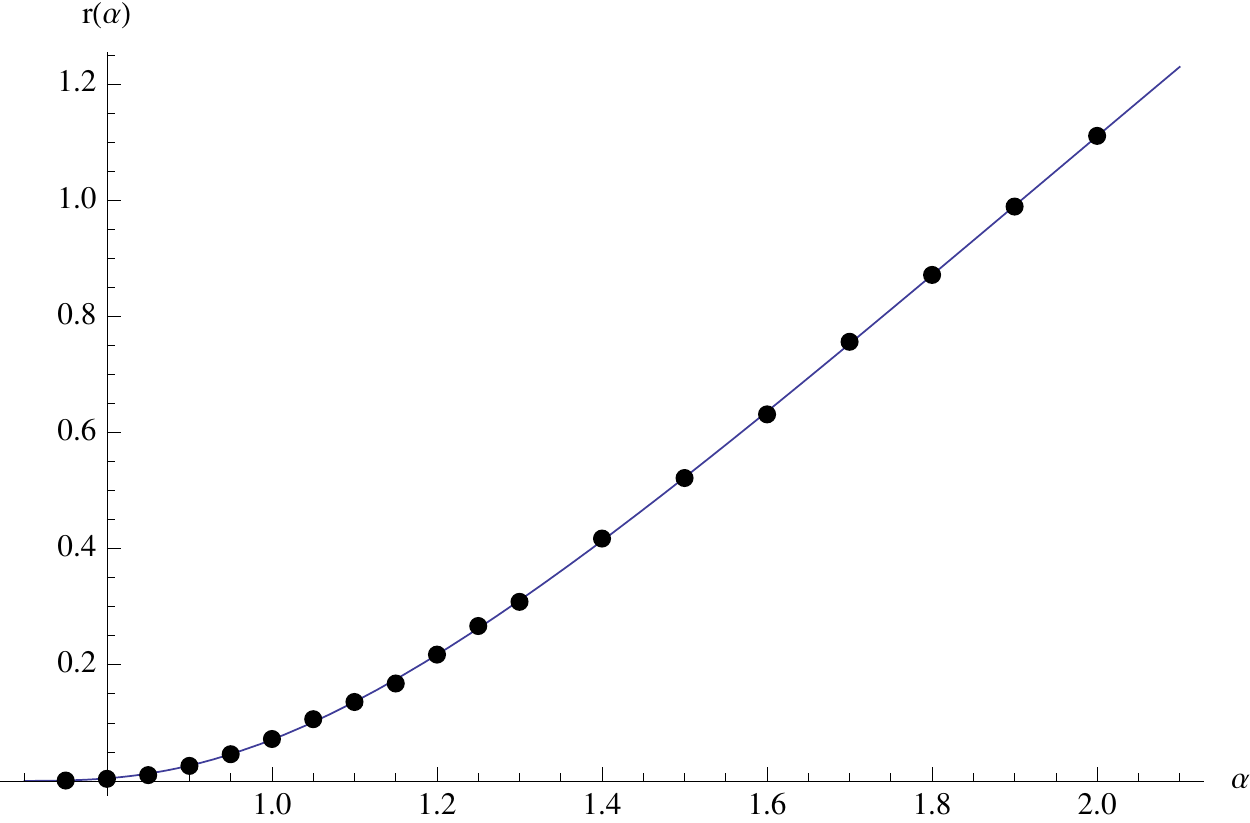} 

\end{center}
\caption{a) The power $p(\alpha)$ of the power law decay for small values of $\alpha$. 
The black dots are the points obtained by MCMC simulation, the dotted red curve is the best fit of the form 
\eqref{formOfP}. The solid blue line is a linear extrapolation of the values of $p(\alpha)$ for $\alpha \leq 0.5$. 
The dotted green line is the 'universal' critical exponent $p_c = 0.25$.
b) The rate $r(\alpha)$ of the exponential decay for large values of $\alpha$. The black dots are obtained by MCMC simulation,
the blue line is the best fit of the form \eqref{formOfR}.
}
\label{fig5}
\efig

There is, however, a different reasoning that leads to a different result. In the $XY$-model
it is known that, to first approximation, the power $p(\beta)$ with which correlations decay 
is linear in $\beta$ \cite{KT73}. In addition, 
the critical exponent $p(\beta_c)$ is known to be $1/4$, and is thought to be universal for 
all systems that exhibit a Kosterlitz-Thouless phase transition. 
The argument leading to $p(\alpha_c) = 1/4$
is based on the analysis of vortex unbinding, see \cite{ChLub2000}, 
Chapter 9, for a summary. As discussed above, we do not 
have any obvious vortices in SRP, and the phase of power law decay is characterized 
by the presence of non-trivial topological objects rather than their absence. We nevertheless assume that 
the general theory applies. Since we can trust our numerics really only for small $\alpha$ and not for those 
close to the alleged phase transition, we linearly extrapolate all measured points for values of 
$\alpha \leq 0.5$. The result is $p_{\rm lin}(\alpha) \approx -0.019 + 0.415 \alpha$,
shown in Figure \ref{fig5} a), as a dashed line. It intersects the line $p = 0.25$ at 
$\alpha_c \approx 0.648$, which is a rather different prediction from the one obtained in the previous paragraph,
and at the moment it is not clear why we should give it any credit.  
We will however soon see that it fits extremely well with a similar prediction obtained from applying the theory of 
the Kosterlitz-Thouless phase transition to simulation results for large $\alpha$.  

By Proposition \ref{large alpha} we know that $\nu(K)$ decays exponentially for 
large enough $\alpha$. The decay rate itself and its dependence on $\alpha$ have
to be estimated numerically. To obtain a numerical approximation of $K \mapsto \nu(K)$, 
we average the fraction of points in cycles 
longer than $K$ over $10000$ samples, for all $1 \leq K \leq 2000$, 
with $10$ full sweeps between each two samples. We do this for $0.75 \leq \alpha \leq 1.3$ 
in intervals of $0.05$, and for $1.4 \leq \alpha \leq 2$ in intervals of $0.1$. 
For $\alpha \geq 0.9$, cycles are shorter than $500$ with overwhelming probability, and it is therefore 
safe to use the system size $N = 500^2$. We focus on those points where  
$10^{-3} \leq \nu(K) \leq 10^{-6}$, in order to make sure that we are on the one hand in the asymptotic 
regime, and that on the other hand we do not yet see effects due to the finite system size and the finite simulation
time. By determining the slope of $K \mapsto - \ln(\nu(K))$ in this range of $K$, we obtain the exponential decay rate $r(\alpha)$, i.e.\ 
\be \label{nu}
\nu(K) = \bbP(\ell_x > K) \sim \exp(-r(\alpha) K ).
\ee
Assuming that the above relation is exact for all $K>0$, we obtain the correlation length 
$\xi(\alpha) = \bbE(\ell_x) = 1/r(\alpha)$ of the system at parameter $\alpha$. 
Figure \ref{fig5} b) shows the sampled values of $r(\alpha)$ as black dots. 
By analogy to the general theory of the Kosterlitz-Thouless transition \cite{ChLub2000},
$r(\alpha)$ is expected to be of the form 
\be \label{formOfR}
r(\alpha) = c \exp( - b / |\alpha - \alpha_c|^{1/2}),
\ee
with $\alpha_c$ being the critical temperature. Fitting a curve of the form \eqref{formOfR} to the 
data yields an extremely good result, shown in Figure \ref{fig5} b). The parameters are 
$c \approx 20.99$, $b \approx 3.434$, and $\alpha_c \approx 0.636$. 
The latter is remarkably close to the value $0.648$ obtained by linear extrapolation of 
the decay exponent for small $\alpha$, especially considering that these quantities depend very 
sensitively on small deviations of the data. It seems therefore evident that planar SRP
undergo a KT phase transition near $\alpha = 0.64$. In contrast, a glance at Figure \ref{fig3} shows that  
the phenomenology of a typical SRP changes most between $\alpha = 0.7$ and 
$\alpha = 0.8$, which is when the bubbles appear. 

There a roughly analogous situation in the XY model, 
where a maximum of the specific heat at temperature $T > T_c$ appears due to the entropy 
associated with the vortex pair unbinding, see \cite{ChLub2000}, Figure 9.4.3.  
In our model, the specific heat is connected 
to the derivative of the expected jump length with respect to $\alpha$: fixing $x \in X_L$, 
using the form \eqref{energy2} of the jump energy and translation invariance, 
we get the 'average energy per particle' as 
\[
\langle E \rangle / N = - \frac{1}{N} \partial_\alpha \ln Z_N =  \bbE_N(|\pi(x)-x|^2). 
\]
Therefore the 'specific heat capacity' is given by 
\[
C/N = \frac{\alpha^2}{N} \partial^2_\alpha \ln Z_N = - \frac{1}{N} \alpha^2 \partial_\alpha \bbE_N(|\pi(x)-x|^2).
\]
We have put the thermodynamic terminology in inverted commas since, as we discussed above, it is not easy 
to give a true thermodynamic interpretation of our system. Nevertheless, we can see the analogy with the XY-model: 
When bubbles start to form at some $\alpha > \alpha_c$, the expected jump 
length starts to change more rapidly, since a jump length of zero is not possible for a point in a cycle. For 
even smaller $\alpha$, a saturation sets in and the average jump length does not change too much any more. The 
critical value $\alpha_c$ is already well within this saturation regime. Thus, 
as in the XY model, the 'specific heat' is greatest for values of $\alpha$ that are above the KT transition point. 

\section{Fractal dimension} \label{section fractal}

In this section, we investigate geometrical properties of long SRP cycles. 
We give numerical evidence that long SRP cycles are random fractals with 
an almost sure fractal dimension $d(\alpha)$ 
depending on the parameter $\alpha$. $d(\alpha)$ 
shows an interesting change of behavior near the transition point $\alpha \approx 0.65$. 

As usual, the quantity that we approximate numerically is the 
box counting dimension, rather than the (numerically inaccessible) Hausdorff dimension. 
Let us recall the definition of the box counting dimension: 
we tile $\bbR^d$ with 
cubes of side length $\eps$, and for compact  $A \subset \bbR^d$ we write $N_A(\eps)$ 
for the minimal number of cubes needed to cover $A$. The box counting dimension of $A$ is defined as 
\be \label{BCD}
d_{\rm Box}(A) = \lim_{\eps \to 0} \frac{\ln N_A(\eps)}{\ln (1/\eps)}
\ee
if this limit exists. 
It is known that the box counting dimension is always at least as large as the 
Hausdorff dimension, and there are simple (but unnatural) examples where it is significantly larger. 
For all 'natural' fractals that I am aware of, the box counting dimension 
agrees with the Hausdorff dimension. See \cite{Falc03} for more background on fractal geometry. 

For a numerical approximation of the box counting dimension, we sample from a fully thermalized random
permutation of side length $L$, and retain the points that belong to the longest cycle as our set $A$.
We take $m$ samples of $N_A(\eps)$ for logarithmically equidistant values of $\eps$. Explicitly, 
we cover the volume $\Lambda$ with $n_j^2$ boxes of side length $r_j = L / n_j$, for $j=0, \ldots, m-1$, with $n_j > n_{j-1}$ 
for all $j$. $n_0$ must be chosen such that the smallest 
box size $r_0$ is large enough to avoid seeing the finite grid size, and $n_j$ is the nearest integer to 
$n_0 c^j$, for some $c<1$.  
For quantitative accuracy of the results below, it is crucial that $n_j$ is integer; otherwise boundary boxes will be 
of different size than bulk boxes, giving small but noticeable errors. We count the number $N_A(r_j)$ 
of boxes that contain at least one point of the longest cycle. If $A$ does have a box counting dimension, 
we expect the points $\{ (\ln (1/r_j), \ln N_A(r_j)): 1 \leq j \leq m \}$ to lie approximately on a straight line, 
whose slope is an approximation to $d_{\rm Box}(A)$. This is confirmed by Figure \ref{fig6} a), which shows 
the results of the above procedure for $L = 4000$, $c = 0.95$, and $n_j=4000$. 
The latter means that the minimal box size is $1$, which is of course too small; 
finite size effects are clearly visible in Figure \ref{fig6} a). On the other hand,
the values for $n<1000$, i.e. box side lengths greater than $4$, are nicely linear. 
The analogue of Figure \ref{fig6} for other values of $\alpha$ looks very similar.
We have thus very good reason to assume that long cycles of SRP are indeed fractals. 

For values of $\alpha$ greater than $0.7$, long cycles become
increasingly rare, and are virtually absent above $\alpha = 0.85$. This is expected in view of the 
previous section, but is inconvenient for measuring their box counting dimension. 
To avoid this problem we use metastability to our advantage: 
we take as initial condition the spatial permutation 
with precisely one periodic cycle of minimal length, with nearest neighbor jumps. Thermalizing this initial 
condition will force a cycle of at least length $L$ to be in the system at all times, due to the strong metastability
of SRP. 

\bfig
\begin{center}
{\tiny a)}
\includegraphics[width=0.45 \textwidth]{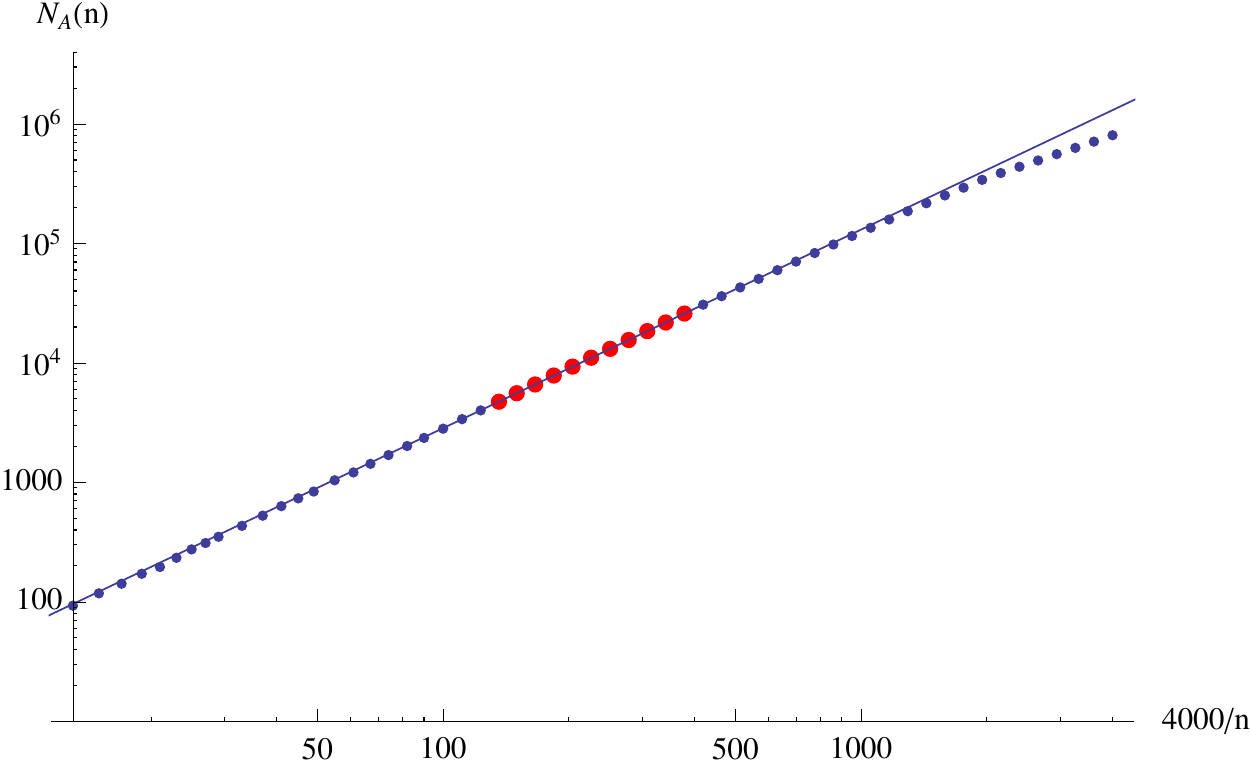} 
{\tiny b)}
\includegraphics[width=0.45 \textwidth]{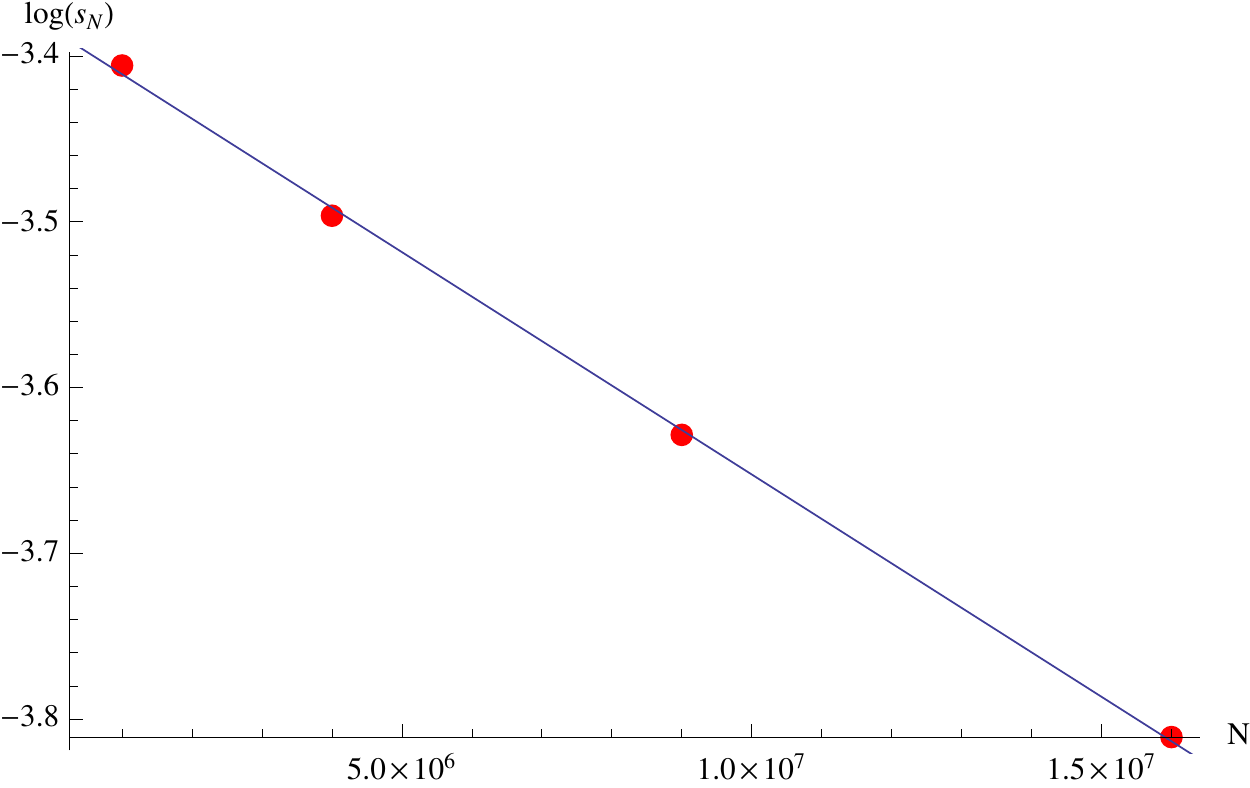} 
\end{center}
\caption{a) Double logarithmic plot of the number $N(\eps)$ of boxes needed to cover a long SRP cycle at
$\alpha = 0.5$, as a function of the inverse of the box side length $\eps$. 
The linearity strongly suggests the existence of a box counting dimension. The straight line is fitted to the 
slightly larger red points. b) Logarithmic plot of the 
empirical standard deviation $s_N$ of $500$ samples for the box counting dimension, as a function of system size $N$. 
}
\label{fig6}
\efig

The fractal dimension, given by the slope of the linear approximation as discussed
above, does of course depend on the sample that we take. However, the dependence 
decreases with large system size: Figure \ref{fig6} b) shows a logarithmic plot of 
the standard deviation of a sample of $500$ measurements of the box counting dimension, for $\alpha = 0.5$ and system sizes 
$1000^2$, $2000^2$, $3000^2$ and $4000^2$. The minimal box side length was chosen to be $4$. 
The standard deviation clearly decreases exponentially with the system size; 
linear interpolation suggests a decay like $0.034 \exp(-2.62 \times 10^{-8} N)$. Based on this numerical evidence,
we conjecture that the box counting dimension of long SRP cycles is an almost sure property in the scaling limit. 

An accurate quantitive study of the box counting dimension as a function of $\alpha$ 
is tricky. The reason is that even though the curve in Figure \ref{fig6} a)
looks very much like a straight line, the slope of any local linear approximation 
depends to some degree on the points chosen for the linear regression. 
The regression line that is shown in Figure \ref{fig6} a) is fitted to the values between 
$n_{\rm min}=136$ and $n_{\rm max}=378$, 
which means box sizes of roughly $10 \times 10$ up to 
$30 \times 30$; we thus increase the box volume by one order of magnitude.
Figure \ref{fig7} a) shows the resulting slope for various values of $n_{max}$, 
ranging from $n_{\rm max} = 2000$ (minimal box side length $2$) to $n_{\max} = 80$ 
(minimal box side length $50$), with $n_{\rm min} \approx n_{\rm max}/3$. We 
can see that for minimal box side lengths between $25$ and $70$, the slope does not change significantly,
while for smaller minimal box sizes, we start to see finite size effects. This is a problem, since 
systematically sampling from a SRP of side length $4000$ is very expensive in terms of computing time, 
especially so for high values of $\alpha$ where time correlations decay extremely slowly. 
On the other hand, at side length $1000$, where systematic computations are still feasible, 
effects from the fact that $L$ is finite set in before we lose effects from the discrete lattice.

\bfig
\begin{center}
{\tiny a)}
\includegraphics[width=0.45 \textwidth]{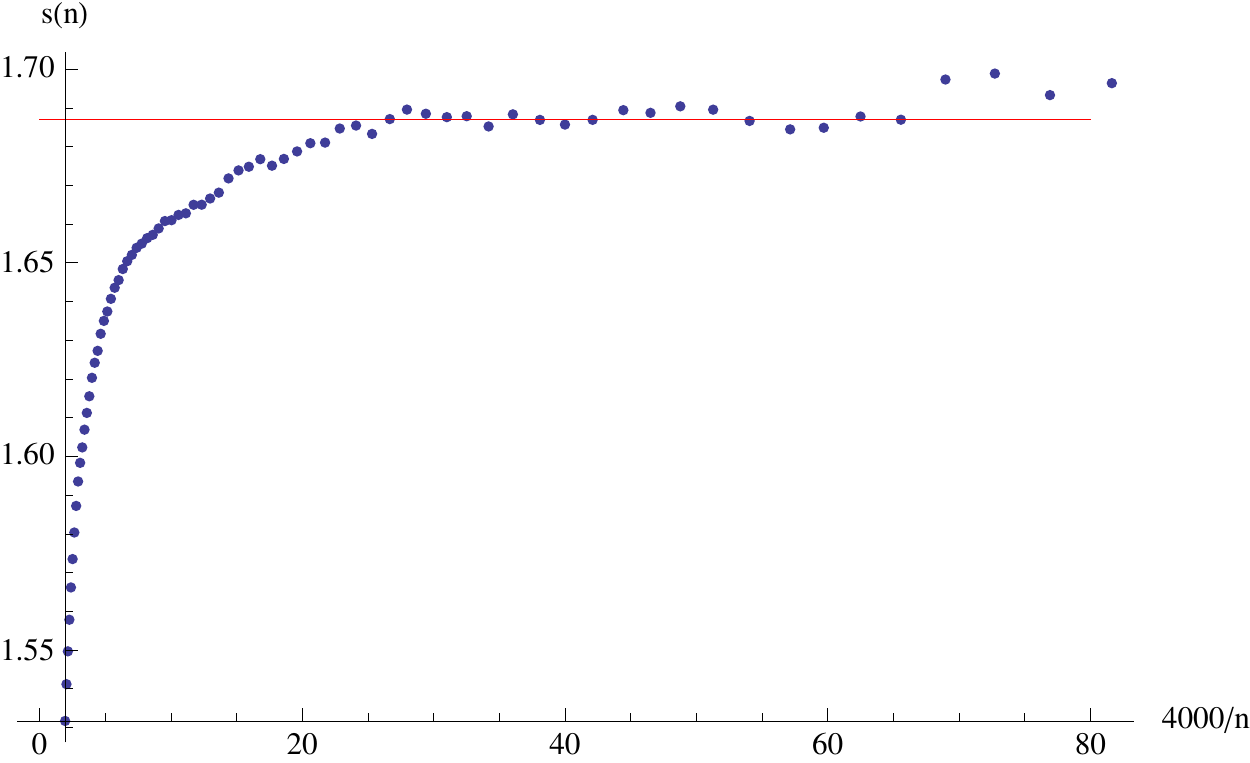} 
{\tiny b)}
\includegraphics[width=0.45 \textwidth]{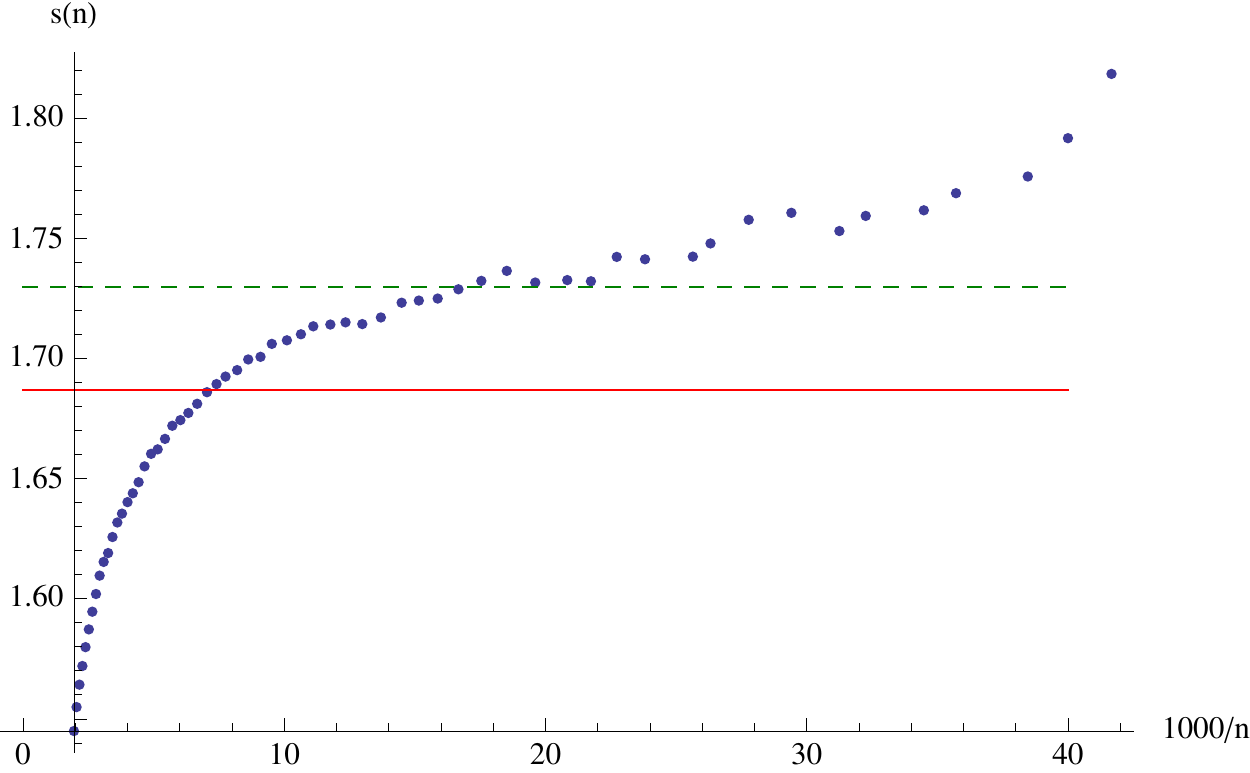} 
\end{center}
\caption{a) A linear function is fitted to the points $(-\ln \eps_{j,k},\ln N(\eps_{j,k}))_{j=0,\ldots,20}$, for the 
largest cycle of a SRP with $\alpha = 0.5$ and $L=4000$, where $\eps_{j,k} = n_k/4000 \times 0.95^j$; 
the slope of this local linearization is plotted as a function of the 
minimal box size $4000/n_k$. Effects from the discreteness of the lattice are visible below a minimal box side length 
of 25. In the region where the slope is approximately constant, the box counting dimension can be 
estimated as $1.687$, shown as a red line. 
b) The same graph, but for a SRP of side length $L=1000$ and $\alpha = 0.5$. Finite volume effects 
set in before the box size $25$ is reached. The line $x = 1.687$ is shown in red; however, by the finite 
standard deviation of the numerical box counting dimension for finite volume SRP this does not mean too much. 
A better guess might be $x = 1.73$, shown as a dashed green line; ultimately, however, $L=1000$ seems to be 
too small for an accurate numerical determination of the box counting dimension. 
}
\label{fig7}
\efig

This is demonstrated in Figure \ref{fig7} b), which shows the local slopes for a SRP of side length $1000$ at 
$\alpha = 0.5$, with $n_{\rm min} \approx n_{\rm max}/3$. So with this side length, we cannot hope for 
very accurate quantitative results on $\alpha \mapsto d_{\rm Box}(\alpha)$; these would require much more careful 
and extensive numerical studies. 
Fortunately, the qualitative  behavior of $d_{\rm Box}(\alpha)$ is interesting 
as well and, as some numerical testing showed, 
does not depend too much on our choice of $n_{\rm max}$. We chose a minimal box size of $8$ as a 
compromise between finite system size and finite lattice spacing errors. 

\bfig
\begin{center}
\includegraphics[width= \textwidth]{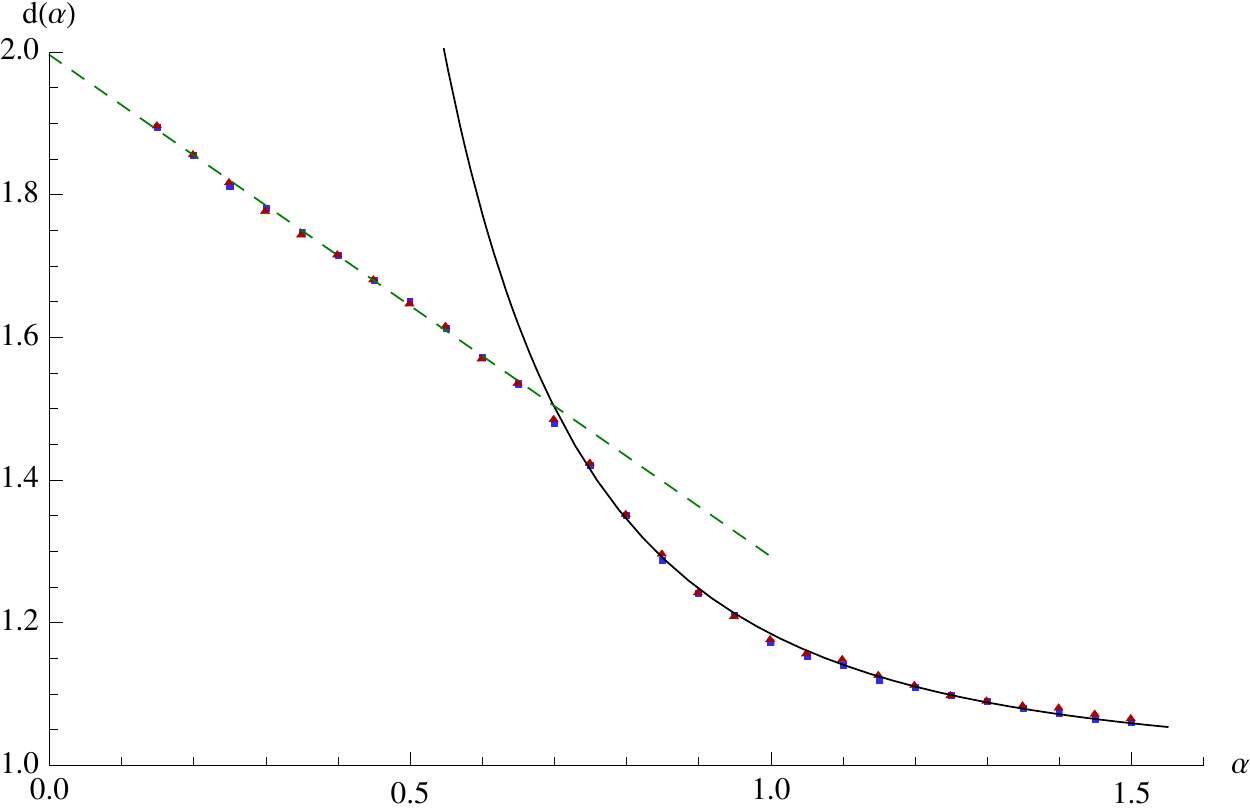} 
\end{center}
\caption{Box counting dimension $d(\alpha)$ of the longest cycle of 
a SRP, as a function of $\alpha$. Blue boxes 
are measured values for the square grid. Red triangles are measured 
values for the triangular grid. The green dashed and black 
solid line are linear and power law approximations to 
$d_{\rm Box}(\alpha)$ in the respective regimes. }
\label{fig8}
\efig

Figure \ref{fig8} shows, as a function 
of $\alpha$, the box counting dimension $d_{\rm Box}(\alpha)$ 
of long cycles from a SRP on a square lattice and a triangular lattice with $L=1000$.
At each value of $\alpha$ that has been sampled, 
the box counting dimension of the longest cycle was measured 
$500$ times, with $10$ sweeps between
consecutive measurements. $\alpha$ varies in a range of parameters between 
$\alpha = 0.15$ and $\alpha = 1.5$, in 
steps of $0.05$. We do not show values of $\alpha$ smaller than $0.15$, as another finite volume effect makes 
simulations unreliable there: 
often at very low $\alpha$, cycles occupy all available boxes down to the minimal box size of $8$, and thus 
look like they would be space filling. This gives an
erroneous sample value $d = 2.0$. The very slow decay of $\nu(K)$ (see \eqref{nu}) is not yet visible. Thus
the box counting dimension is overestimated by the numerics in this region. 

The most striking features about Figure \ref{fig8} are firstly 
that there is very little difference between 
the results for a square grid and a triangular grid; and secondly 
that there is a regime change around $\alpha = 0.7$. 
For $\alpha \leq 0.7$, $d(\alpha)$ is approximately linear. We fit 
a function of the form 
\be \label{linBCD}
d_{\rm lin}(\alpha) = d_0 - s \alpha
\ee
to the measured values of $d_{\rm Box}(\alpha)$, for all $\alpha \leq 0.65$ and both grids. 
In Figure \ref{fig8} this is shown as a dashed green line. We find $d_0 = 1.996$ and 
$s = 0.705$. Since $\alpha=0$ corresponds to uniform permutations, we expect $\lim_{\alpha \to 0} d(\alpha)=2$, 
and the numerical results confirm this nicely. It is tempting to conjecture that the slope $s$ has the value
of $7/10$. Our numerical accuracy is not high enough to justify significant confidence in such 
a quantitative prediction, but we feel comfortable to conjecture that 
$d(\alpha)$ is a linear function of $\alpha$ for small $\alpha$. 

For $\alpha \geq 0.8$, $d(\alpha)-1$ decays like a power law. Fitting a function of the form 
\be \label{powBCD}
d_{\rm pow}(x) = 1 + b x^{-c}
\ee
to the measured values of $d_{\rm Box}(\alpha)$ for $\alpha \ge 0.75$ 
gives $b \approx 0.184$ and $c \approx 2.806$. The corresponding function is shown in Figure 
\ref{fig8} as a solid black line. The quality of the fit is quite 
remarkable given that there are only two free 
parameters. Other choices of the minimal box size lead to 
slightly different exponents, and so the numerical
values of $b$ and $c$ should not be taken too seriously. However, extensive 
tests have shown that 
the form \eqref{powBCD} remains valid for other choices of the minimal box size. 

This does not rule out the possibility that the 
parameters in $\eqref{powBCD}$ 
depend on the system size $N$, and that in particular the 
parameter $b$ could converge to zero as 
$N \to \infty$. To see why this is possible to happen, consider the case of very
large $\alpha$. Then almost all of the total energy will come from the 
large periodic cycle that was forced
through the system, and jumps longer than the minimal lattice distance will be 
very rare. The model thus becomes very similar to the self-avoiding 
random walk connecting two points of a domain. 
The trace of the walk in this model is strongly believed \cite{DKY12} to 
converge to a straight line in the scaling limit when $\alpha$ large. It is 
therefore plausible that for large $\alpha$, the fractal dimension
of large cycles equals one in the scaling limit.
On the other hand, in \cite{DKY12} it is proved that for small $\beta$, 
the self-avoiding walk becomes space-filling; this is 
certainly not what happens for SRP cycles and shows that in general, the comparison of 
SRP cycles and paths of self-avoiding walks has to be treated with caution. 

\section{The scaling limit} \label{sle}

The results of the previous section lead to the next question: is there a scaling 
limit of long SRP cycles, and if so, what can we say about it? 
Let $C_x = \{\pi^j(x): j \in \bbZ\}$ be an infinitely long cycle 
of a SRP in infinite volume, obtained by suitable conditioning as described before.
Connecting $\pi^j(x)$ and $\pi^{j-1}(x)$ with a 
straight line for all $j \in \bbZ$ produces a piecewise straight random curve 
in $\bbR^2$, with possible corners in the points of $X$. Consider $(C_x-x)/M$ and let $M \to \infty$. 
Optimistically, we would expect the limit (in distribution) 
to exist just as it does in the construction of Brownian motion 
via Donskers theorem, and to be a measure on random curves through $x=0$.
To prove this is an entirely different matter, and we cannot offer any 
contributions to that question. 

If the limit does exist, the results from 
Section \ref{section fractal} suggest that it is a fractal. In particular,
the limiting measure is invariant under scaling. 
Moreover, assuming that the jump energy $\xi$ is symmetric under 
rotations (we will do this from now on), 
we can expect the limiting measure to be invariant under rotations, just as is the case 
for Brownian motion as a limit of the simple random walk. With a somewhat larger 
leap of faith we may claim that the limiting measure (if it exists) 
is conformally invariant, i.e. invariant under all conformal maps. This is 
of course much more than just rotation and scaling invariance, and we expect it 
to be very difficult to show. Supporting evidence is that, by our 
numerical results from Section \ref{section fractal}, the limit seems not to depend
on the chosen lattice. 

The question that immediately arises is whether long cycles of SRP might have some 
connection with curves of the Schramm-L\"owner evolution (SLE). After all, 
the latter are also conformally invariant, and arise as the 
scaling limit of discrete curves, 
such as the loop-erased random walk \cite{LSW04}, or the exploration path of critical
percolation on a triangular lattice \cite{Smi01}.
Since we feel that the only people likely to benefit from the remainder of the 
present section are those already acquainted with SLE, we will
not try to give any details on its theory. We refer to 
\cite{Law05}, the nice chapter by Vincent Beffara in \cite{ENSW12},
and also to \cite{Car05} for introductions to the topic.

Besides conformal invariance in the scaling limit, the second property that 
lattice models converging to SLE must have is the domain Markov property \cite{Schr00}. 
It can be described as follows: Let $\mu_{x,A}$ be a family of 
measures on self-avoiding discrete paths on the lattice $X$, where $A$ is a subset of 
$X$ and $x \in A$. Under $\mu_{x,A}$, the first step of the random path starts from 
$x$, and the path avoids itself and all points in $A$ with probability one. 
If $\mu_{A,x}$, conditional on the first $n$ steps $(x_1, \ldots, x_n)$ of the path, is equal to 
$\mu_{A \cup \{x_1, \ldots,x_n \}, x_n}$, then we say that the family $\mu_{A,x}$ has the domain 
Markov property. 

A variant of the SRP measure \eqref{measure} has this property: 
consider $A \subset X$ such that the complement of $A$ is 
finite, and let $x \in A$.  Write  
\[
\caS_{A,x} = \{ \pi: A^c \cup \{x\} \to X \setminus \{x\}: \pi \text{ is one-to-one}, 
\pi(A^c) \cap A = \{y\} \text{ for some } y \neq x \}
\]
for the set of maps that are permutations on $A^c$ except for one cycle that starts in 
$x$ and ends in some $y \neq x$ in $A$. For $\pi \in \caS_{A,x}$ define 
\be \label{open loop measure}
\bbP_{A,x}(\pi) = \frac{1}{Z_{A,x}} \e{-\alpha H(\pi)},
\ee
where $H$ is given by \eqref{energy}, it is understood that $\pi(z)=z$ for all 
$z \in A \setminus (\pi(A^c) \cup \{x\})$, and the sum in the normalization 
$Z_{A,x}$ goes over $\caS_{A,x}$. In other words, our measure is just like the ordinary
SRP measure except that there is an 'open' cycle going from $x$ to $y$. 
A variant of this measure has been introduced in \cite{Uel06} to make a 
connection between random permutation models and off-diagonal long range order 
for the Bose gas. 

For a path $(x=z_0,z_1,\ldots,z_n=y)$ from $x$ to $y \in A \setminus x$, and
$m \leq n$, we have 
\[
\sum_{\pi \in \caS_{A,x}} \e{-\alpha H(\pi)} 
1_{\{ \pi^j(x) = z_j \,\, \forall j \leq m \} }  = 
\prod_{j=0}^{m-1} \e{-\alpha \xi(\pi^{j+1}(x) - \pi^j(x))} 
\sum_{\pi \in \caS_{A \cup \{z_1, \ldots, z_m\},z_m}} \e{-\alpha H(\pi)},
\]
and thus 
\[
\bbP_{A,x} (\pi^j(x) = z_j \,\,  \forall j > m | \pi^j(x) = z_j \,\, \forall j \leq m) = 
\bbP_{A \cup \{z_1, \ldots, z_m\},z_m} (\pi^{j}(z_m) = z_{m+j} \, \forall j).
\]
This is the domain Markov property. When $x=0$ and $A$ is the union of $\{x\}$ and the 
exterior of a large ball, 
we are in a situation reminiscent of radial SLE; when $X$ is the intersection of a 
regular lattice with a half space, $x$ is on the boundary of that half space, 
and $A^c \nearrow X$, the model is close to 
the ones that produce chordal SLE. The original setting \eqref{measure} might correspond
to the conformal loop ensemble \cite{SW12}, at least for some values of $\alpha$. 

There is a major obstacle to any possibility of SRP cycles being 
equal to SLE curves in the scaling limit: the SLE curves, and the discrete 
curves of any known model leading to them in the scaling limit, are non-self-crossing,
although the former do have self-intersections if the SLE-parameter $\kappa$ is greater than $4$.  
On the other hand, long SRP cycles are always self-crossing, especially so for 
small values of $\alpha$. Of course, it is still possible that the scaling limit 
of an SRP curve describes an SLE-trace as a set. It is also possible that in the scaling 
limit, self-crossings become invisible; after all, jumps are of finite length, and areas 
where many lattice points are already used are difficult to cross for a cycle. This is 
especially plausible for large values of $\alpha$. Nevertheless, the fact that SRP
curves are self-crossing before taking the scaling limit casts some doubt over the 
claim that SRP cycles might be distributed like SLE curves.

One way to make the curves of SRP non-self-intersecting is to 
simply exclude configurations where the straight line between $x$ and $\pi(x)$
crosses the one between $y$ and $\pi(y)$ for some $x$ and $y$. This happens 
automatically if we only allow jumps between nearest neighbors in $X$, i.e.\ 
we put $\xi(x-y) = \infty$ in \eqref{energy} unless $x=y$, or $x,y$ are nearest
neighbors. If we put $\xi(0) = \infty$ as well 
then each permutation has precisely $N$ jumps. Assuming in addition that all 
nearest neighbors have equal distance, the only factor determining the weight of 
a permutation $\pi$ is the number of cycles of length $\geq 3$ that it has: 
each of these can occur either forwards of backwards, while for cycles of length
two this makes no difference. We end up with 
\[
\bbP_N(\pi) = \frac{1}{Z_N} 2^{k(\pi)}
\]
where $k(\pi)$ is the number of cycles longer than $2$ in $\pi$. This is precisely 
the measure on the loops of the double dimer model \cite{KW10}, which has 
recently been shown to be conformally invariant in the scaling limit 
\cite{Ken12}. The SLE parameter for the double dimer model is  
$\kappa = 4$ \cite{KW10, Ken12}.  

When we put $\xi(0) = 0$, we obtain a double dimer model with nonzero
'monomer activity' governed by the jump energy parameter $\alpha$.  
We are not aware of any investigation of such a model, but it seems plausible 
that it would still be conformally invariant, and that the SLE coefficient cannot
be greater than $4$. This fits well with the result \cite{SW12} that the 
random loop ensemble has $\kappa \in (3/8,4]$. ($\kappa \leq 3/8$ may still 
be possible in SRP by using \eqref{open loop measure} 
and thus forcing a long cycle.) 
It appears therefore that when preventing self-crossings in the lattice 
approximation, the limiting curves, if they turn out to be SLE, cannot have a 
parameter greater than $4$. 

SLE usually occurs at criticality. This would make 
the value $\alpha_c$ where the Kosterlitz-Thouless phase transition happens
our strongest candidate for convergence of SRP cycles to SLE curves. 
We have seen in Section \ref{section KT} that $\alpha_c \approx 0.64$. 
Beffara \cite{Bef08} showed that the fractal dimension
of an SLE curve with parameter $\kappa$ is given by $\min(2, 1 + \kappa/8)$.
In Section \ref{section fractal}, we measured the box counting dimension for 
$\alpha = 0.65$ to be $1.534$ for the square grid, and $1.537$ for the triangular
grid. The empirical standard deviations are $0.025$ and $0.019$, respectively. 
This would give $\kappa \approx 4.28$, with standard deviation of approximately 
$0.2$. As discussed in Section \ref{section fractal}, we cannot exclude a systematic
error in the quantitative measurement of the box counting dimension, and the same
goes for the Kosterlitz-Thouless transition point.  Thus it is entirely possible 
that at $\alpha_c$, the SLE parameter is $\kappa = 4$. This value is special
since it is the one below which SLE curves become non-self-intersecting.

On the other hand, SRP behave somewhat like systems at 
criticality for all  $\alpha$ 
in the sense that correlations $\mu(x,y) = \bbP(y \in C_x)$ decay algebraically. For small $\alpha$, this a feature 
shared by all models undergoing a Kosterlitz-Thouless transition. For large $\alpha$,
it is a simple consequence of the fact that we forced a cycle through the system. 
If we were inclined to link SLE behaviour to algebraic decay of correlations instead of 
strict criticality, we might therefore hope for long cycles of SRP to be SLE 
curves at all parameters $\alpha$. Then, using \eqref{linBCD} with $d_0 = 2$ and $s = 7/10$, 
this leads to $\kappa(\alpha) = 8 - \frac{28}{5} \alpha$ for small $\alpha$. 
For large $\alpha$, \eqref{powBCD} gives 
$\kappa(\alpha) \approx 1.5 \alpha^{-2.8}$, albeit with significant 
numerical uncertainty; it is also possible that in truth 
$\alpha \mapsto \kappa(\alpha)$ has a jump at some value of $\alpha$, 
and is equal one above that value; cf.\ the discussion at the end of 
Section \ref{section fractal},

Of course, all of this is relevant only if 
there actually exists a connection between 
(possibly modified) SRP cycles and SLE curves, 
either for all values of $\alpha$ 
or (more modestly) for some values of $\alpha$. 
Whether such a connection does exist is wide open.

\end{document}